\documentclass[12pt,reqno]{article}
\usepackage{geometry}
\usepackage[utf8]{inputenc}
\usepackage{amsfonts}
\usepackage{amsmath}
\usepackage{amsthm}
\usepackage{amssymb}
\usepackage{graphicx} 
\usepackage{tikz}
\usetikzlibrary{external}
\usetikzlibrary{decorations.pathreplacing}
\usepackage{enumitem}
\usepackage{hyperref}
\usepackage{xcolor}
\usepackage{mathtools}
\usepackage{relsize}
\usepackage{caption}

\usepackage{hyperref}
\hypersetup{
    colorlinks=true,
    linkcolor=blue,
    filecolor=magenta,      
    urlcolor=green,
}

\newtheorem{thm}{Theorem}[section]

\newtheorem{lem}[thm]{Lemma}

\newtheorem{cor}[thm]{Corollary}

\theoremstyle{definition}
\newtheorem{defn}[thm]{Definition}

\newtheorem{ex}[thm]{Example}

\newtheorem{rem}[thm]{Remark}

\numberwithin{equation}{section}
\newtheorem{op}[thm]{Problem}

\usepackage{algorithm}
\usepackage[noend]{algpseudocode}
\algnewcommand\algorithmicinput{\textbf{Input:}}
\algnewcommand\algorithmicoutput{\textbf{Output:}}
\algnewcommand\Input{\item[\algorithmicinput]}%
\algnewcommand\Output{\item[\algorithmicoutput]}

\renewcommand\bar\overline

\newcommand{\vol}{\mathrm{Vol}}
\newcommand{\nvol}{\mathrm{NVol}}

\usepackage[english]{babel}
\usepackage[autostyle, english = american]{csquotes}
\MakeOuterQuote{"}

\def\0{\mathbf{0}}
\newcommand{\bba}{\mathbf{a}}
\newcommand{\bbb}{\mathbf{b}}

\newcommand{\bbc}{\mathbf{c}}
\newcommand{\bbe}{\mathbf{e}}

\newcommand{\bbr}{\mathbf{r}}

\newcommand{\bbu}{\mathbf{u}}
\newcommand{\bbv}{\mathbf{v}}
\newcommand{\bbw}{\mathbf{w}}
\newcommand{\bbx}{\mathbf{x}}
\newcommand{\bby}{\mathbf{y}}
\newcommand{\bbz}{\mathbf{z}}

\newcommand{\conv}{\mathrm{conv}}

\newcommand{\ncone}{\mathrm{ncone}}

\newcommand{\Z}{\mathbb Z}
\newcommand{\N}{\mathbb N}

\newcommand{\R}{\mathbb R}

\newcommand\cov{\mathrel{\ooalign{$\prec$\cr
  \hidewidth\raise0.0ex\hbox{$\cdot\mkern0.9mu$}\cr}}}

\newcommand{\calC}{\mathcal{C}}

\newcommand{\calF}{\mathcal{F}}

\newcommand{\ads}{\mathbf{AdS}}
\newcommand{\sym}{\Delta^0}

\newcommand{\nn}{[n,\Bar{n}]}

\providecommand{\keywords}[1]
{
  \small	
  \textbf{\textit{Keywords---}} #1
}

\title{Counting Lattice Points in Generalized Permutohedra From A to B}

\author{Warut Thawinrak\thanks{Beijing International Center for Mathematical Research, Beijing China, \texttt{warutthawinrak@gmail.com}}}
\date{}

\begin{document}

\maketitle
\thispagestyle{empty}

\begin{abstract}
    We derive a formula for the number of lattice points in type B generalized permutohedra, providing a concise alternative to the formula obtained recently by Eur, Fink, Larson, and Spink as a result from a study of delta-matroids. Our approach builds upon the existing framework and techniques introduced by Postnikov in his work on type A generalized permutohedra, a family of polytopes interconnected with many mathematical concepts such as matroids and Weyl groups. In particular, we express the number of lattice points in type B generalized permutohedra in terms of Postnikov’s notion of G-draconian sequences, from which their Ehrhart polynomials and volume formula follow as consequences.
\end{abstract}



\section{Introduction}

Type A generalized permutohedra form a rich class of polytopes that arise naturally in contexts such as the theory of matroids and Weyl groups, lying at the intersection of several areas of mathematics. Their geometric and combinatorial properties have been the subject of extensive study. In particular, Postnikov, Reiner, and Williams \cite{Postnikov2005, PostnikovEtal2008} established a comprehensive framework for type A generalized permutohedra, leading to extensive investigations of their lattice points, facial structures, various combinatorial interpretations, and applications. 

Type B generalized permutohedra, by contrast, form a broader class of polytopes encompassing those of type A but remain comparatively less explored. Given the depth and maturity of the results and techniques in the type A setting, it is natural to seek generalizations or, in some instances, direct adaptations of these tools to the type B context, thereby uncovering properties that both parallel and extend those known for type A. In this paper, we employ Postnikov’s framework \cite{Postnikov2005} for type A generalized permutohedra to count the lattice points in type B generalized permutohedra. Specifically, by viewing a type B generalized permutohedron in each orthant as a type A generalized permutohedron, we can apply Postnikov's approach to derive analogous lattice-point enumerations via polytopal subdivisions into cells that correspond bijectively to both their lattice points and the so-called \emph{$G$-draconian sequences}. This yields a concise alternative to the formula recently obtained by Eur, Fink, Larson, and Spink \cite[Theorem A]{Euretall2024} from their study of delta-matroids. As a consequence, we also obtain formulas for their Ehrhart polynomials and volume. Moreover, this approach suggests that many aspects of type B generalized permutohedra may be effectively analyzed using existing notions and techniques developed for their type A counterparts.

\subsection*{Paper Organization} We begin by introducing  generalized permutohedra, along with a review of existing techniques for enumerating lattice points in these polytopes. We then develop a method for realizing part of a type B generalized permutohedron as a type $A$ generalize permutohedron. Using this realization, we show how the number of lattice points in type $B$ generalized permutohedra can be computed in terms of the so-called $G$-draconian sequences. Lastly, we propose questions for future work on how the combinatorial properties of type B generalized permutohedra might be deduced from those of their type A counterparts.  

\section{Preliminaries}
\subsection{Polyhedra}

A \emph{polyhedron} $P$ in $\R^n$ is the solution to a finite system of linear inequalities, that is, 
\begin{align}\label{eq:polyhedra-inq}
    P = \left\{\bbx=(x_1, \dots, x_n) \in \R^n\ \Big\vert \  \bba_i \cdot \bbx \leq b_i \text{ for } i \in I\right\}
\end{align}
for some $\bba_i \in \R^n$ and $b_i \in \R,$ where the dot $\cdot$ is the usual dot product and $I$ is a finite set of indices. The dimension of $P$, denoted by $\dim(P),$ is defined to be the dimension of $\mathrm{aff}(P)$ the affine span of $P$.

A \emph{polytope} $P$ is a bounded polyhedron. By the Minkowski-Weyl Theorem \cite{Minkowski1968, Weyl1934}, we can equivalently define a polytope in $\R^n$ as the convex hull of finitely many points in $\R^n,$ i.e.,
\[P = \conv(\bbx_1, \dots, \bbx_k) := \{\lambda_1\bbx_1 +\cdots + \lambda_k\bbx_k\ | \ \lambda_1 + \cdots + \lambda_k = 1, \lambda_i \geq 0 \text{ for all } i \in [k]\}.\]
We denote by $\vol(P)$ the volume of $P$ with respect to the lattice $\Z^n \cap \mathrm{aff}(P)$ in the affine span of $P.$ The \emph{normalized volume} of a $d$-dimensional polytope is defined to be $\nvol(P):= d!\vol(P).$

For two nonempty polytopes $P_1, P_2$ in $\R^n$, the \emph{Minkowski sum} of $P_1$ and $P_2$, denoted by $P_1 + P_2$,  is the set $\{\bbx_1 + \bbx_2 \in \R^n \,|\, \bbx_1 \in P_1, \bbx_2 \in P_2\}$. If $U$ and $V$ are the set of vertices of $P_1$ and $P_2$, respectively, then $P_1+ P_2 = \conv(\bbu + \bbv \ | \ \bbu \in U, \bbv \in V)$. This implies that a Minkowski sum of polytopes is a polytope. The \emph{Minkowski difference} of $P_2$ in $P_1$, denoted by $P_1 - P_2$, is the set $\{\bbx \in \R^d \,|\, \bbx + P_2 \subseteq P_1\}$. Since the vectors that translate $P_2$ to lie in $P_1$ form a polytope, it follows that a Minkowski difference of two polytopes is also a polytope. It is important to note that, in general, the Minkowski difference on polytopes is neither a commutative nor an associative operator. For example, while $(P_1 + P_2) - P_2 = P_1$ always holds, the expression $(P_1 - P_2) + P_2$ may not equal $P_1$, or may even be undefined if the difference $P_1 - P_2$ is empty. Thus, it is crucial to clearly specify the order of the sum (and difference). Given nonempty polytopes $P_i$ in $\R^n$ and signs $\delta_i \in \{1, -1\}$ for $i \in [m]$, we define the Minkowski sum
\[\sum^{m}_{i = 1}\delta_i P_i := Q_1 - Q_2 \text{ where } Q_1 := \sum_{\delta_i = 1} P_i \text{ and } Q_2 := \sum_{\delta_i = -1} P_i.\]

 A subset $F$ of a polytope $P \subset \R^n$ is said to be a \emph{face} of $P$ if there exists a hyperplane $H$ such that $P$ lies on one side of $H$ and $F = P\cap H.$ A face of dimension $\dim(P) - 1$ is called a \emph{facet}, a face of dimension $1$ is called an \emph{edge}, and a face of dimension 0 is called a \emph{vertex}. The partially ordered set $\calF(P)$ of all faces of $P$ ordered by inclusion is called the \emph{face poset} of $P$.
 
 A \emph{cone} $\sigma$ is a polyhedron defined by a system of homogeneous linear inequalities, i.e. inequalities of the form $\bba\cdot \bbx \leq 0$. Given a nonempty face $F$ of a polytope $P \subset \R^n$, the \emph{normal cone} of $P$ at $F$ is the set
\[\ncone(F,P) := \{\bbc \in \R^n \ | \ \bbc\cdot \bbx \geq \bbc \cdot \bby \text{ for all } \bbx \in F \text{ and all } \bby \in P\},\]
that is, $\ncone(F,P)$ is the set of all $\bbc \in \R^n$ such that $\bbc \cdot \bbx$ attains maximum value at $F$ over all points in $P$. The \emph{normal fan} of $P$, denoted $\Sigma(P)$, is the set of normal cones of $P$ at all of its nonempty faces. We say that a polytope $Q$ is a \emph{deformation} of another polytope $P$ if $\Sigma(Q)$ is a coarsening of $\Sigma(P)$, i.e., every cone in $\Sigma(Q)$ is a union of cones in $\Sigma(P)$. A deformation $Q$ of $P$ can be obtained by parallel translations of the facets of $P$.

A (\emph{polyhedral}) \emph{subdivision} of a polytope $P$ is a collection $\calC$ of polytopes of dimension $\dim(P)$, called \emph{cells}, such that $P$ equals the union of all cells in $\calC$, and any two cells intersect only at their common face.

A \emph{lattice point} is a point whose coordinates are integers. A polytope is said to be  \textit{integral} if all of its vertices are lattice points. For a polytope $P$ in $\R^n$ and a non-negative integer $t,$ the $t^{\mathrm{th}}$-dilation $tP$ is the set $\{tx \,|\, x \in P\}.$ We define
$i(P,t) := |\Z^n \cap tP|$ to be the number of lattice points in the $t^\mathrm{th}$-dilation $tP.$ It follows from Ehrhart theory \cite{Ehrhart1962} that when $P$ is an integral polytope, the function $i(P,t)$ is a polynomial in $t$ of degree $\dim(P)$. We call $i(P,t)$ the \emph{Ehrhart polynomial} of $P$. It encodes various geometric and combinatorial properties of $P$: its leading coefficient equals $\vol(P)$, the second coefficient equals half of the sum of its facets' areas (with respect to the lattice in the affine span of each facet), and its constant term is one.

Two integral polytopes $P, Q$ such that $P \subset \R^n$ and $Q \subset \R^m$ are said to be \emph{integrally equivalent} if there exists an invertible affine transformation from $\mathrm{aff}(P)$ to $\mathrm{aff}(Q)$ that preserves the lattice points in the two polytopes. When two integral polytopes are integrally equivalent, they have the same face poset, volume, and Ehrhart polynomials.

\subsection{Type A generalized Permutohedra}

In this section, we introduce the family of polytopes known as \emph{type A generalized permutohedra}, following the notation and framework established by Postnikov in \cite{Postnikov2005}. For a more comprehensive treatment beyond what is presented here, we refer the reader to \cite{Postnikov2005}.

A \emph{type A generalized permutohedron} $P$ in $\R^n$ is a polytope whose edges are parallel to $\bbe_i - \bbe_j$ for some $1 \leq i < j \leq n,$ where $\bbe_i$ denotes the standard basis vector in $\R^n.$ We note that it is also known simply as a \emph{generalized permutohedron}. We include ``type A'' in the name to emphasize the fact that the edge directions of these polytopes are parallel to some vectors in type A positive root system $\{\bbe_i - \bbe_j\ | \ 1 \leq i < j \leq n\}$. 

Let $\bbw = (w_1, \dots, w_n) \in \R^n$ be a point satisfying $w_1 > \cdots > w_n \geq 0.$ The \emph{permutohedron} $\Pi(\bbw)$, defined as the convex hull of all permutations of the coordinates of $\bbw$, is an example of a type $A$ generalized permutohedron. See Figure \ref{fig: permutohedra} for examples of permutohedra and other type A generalized permutohedra. The normal fan of any permutohedron, called the \emph{braid fan} and denoted by $\Sigma_{A_{n-1}}$, is composed of maximal cones defined by the chambers in the arrangement of the hyperplanes
\[H_{i,j} = \{(c_1,\dots, c_n) \in \R^n \ |\ c_i - c_j = 0\} \text{ for all } 1 \leq i < j \leq n,\]
known as the \emph{braid arrangement}. We note that one can alternatively define a type A generalized permutohedron as any deformation of a permutohedron, i.e., a polytope whose normal fan coarsens the braid fan. 

The family of type A generalized permutohedra is a widely studied class of polytopes that interconnect with various combinatorial objects, such as matroids and Weyl groups. Beyond permutohedra, many well-known polytopes in the literature can be realized as type A generalized permutohedra, including zonotopes, associahedra, cyclohedra, and the Pitman-Stanley polytopes (see \cite[Section 8]{Postnikov2005}). Note that every type A generalized permutohedron $P \subset \R^n$ lies on a hyperplane $x_1+ \cdots + x_n = a$ for some real number $a.$ For instance, the permutohedron $\Pi(w_1, \dots, w_n)$ lies on the hyperplane $x_1 + \cdots + x_n = w_1 + \cdots + w_n.$ Thus, the dimension of $P$ is at most $n-1.$ 

\begin{figure}[h!]
    \centering
    \includegraphics[width=0.9\linewidth]{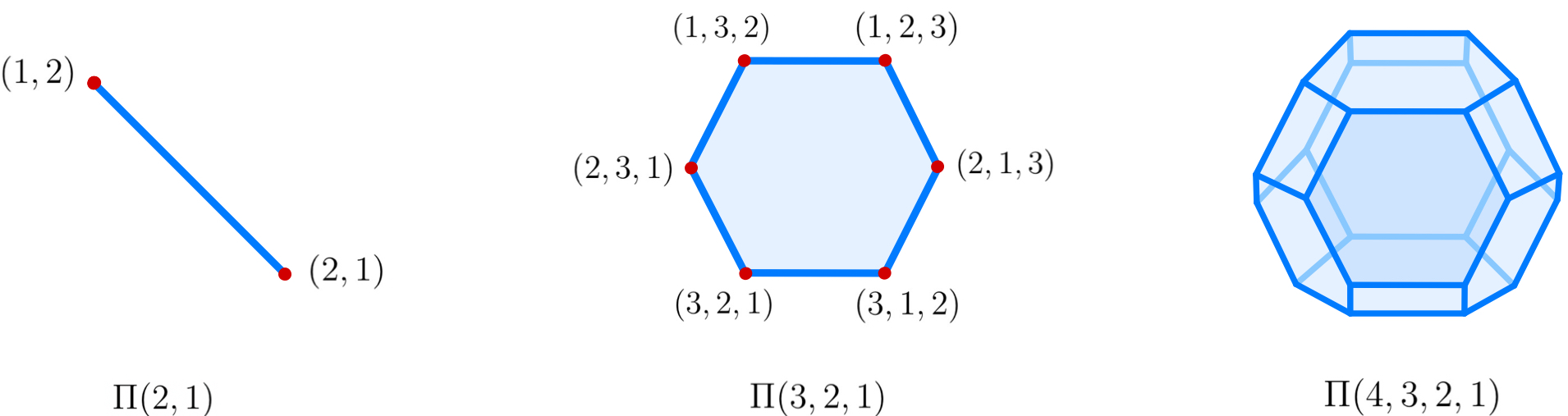}
\end{figure}
\begin{figure}[h!]
    \centering
    \includegraphics[width=0.9\linewidth]{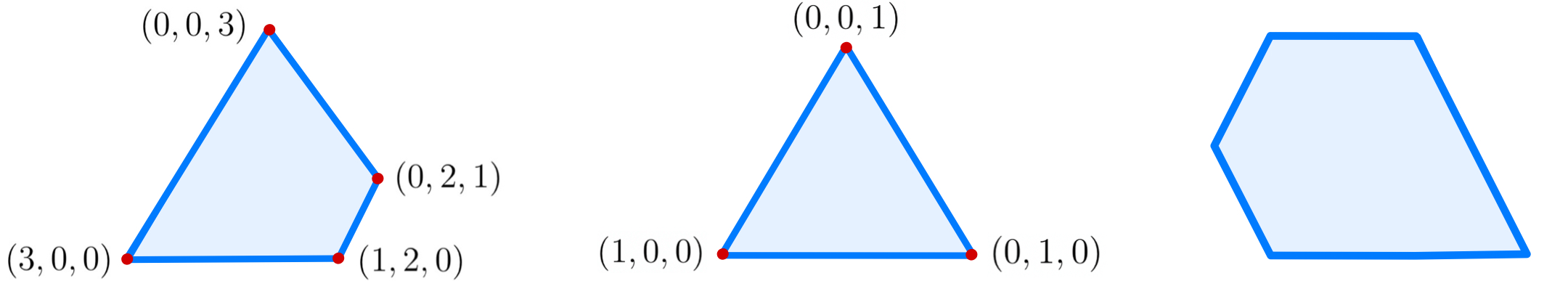}
    \caption{Examples of type A generalized permutohedra}
    \label{fig: permutohedra}
\end{figure}

For a nonempty subset $I \subseteq [n],$ we define $\Delta_I$ to be the simplex $\conv(\bbe_i\ | \ i \in I)$. Postnikov \cite[Proposition 6.3]{Postnikov2005} and Ardila, Benedetti, and Doker \cite[Proposition 2.3]{Ardila2010} show that every type A generalized permutohedron can be written as a Minkowski sum (and difference) of simplices. The following lemma states this more precisely.

\begin{lem}[\cite{Postnikov2005},\cite{Ardila2010}]\label{lem: type-A-Minkowski-sum}
    For every type A generalized permutohedron $P \subset \R^n$, there exists a partition of $2^{[n]}\backslash \{\emptyset\}$ into two disjoint subsets $M_1$ and $M_1$ such that
\begin{align}\label{eq: typeA-Minkowski-sum}
    P + \sum_{I \in M_1}y_I\Delta_I = \sum_{I \in M_2}y_I\Delta_I
\end{align}
for some $y_I > 0$ for all $I \in M_1$ and some $y_I \geq 0$ for all $I \in M_2.$ Moreover, the coefficients $y_I$ in \eqref{eq: typeA-Minkowski-sum} are unique. In particular, there exists a sequence of real numbers $(y_I)_{I \in 2^{[n]}}$ such that 
\[P = \sum_{I \in 2^{[n]}\backslash \{\emptyset\}}y_I\Delta_I.\]
\end{lem}

\begin{figure}
        \centering
        \includegraphics[width=1\linewidth]{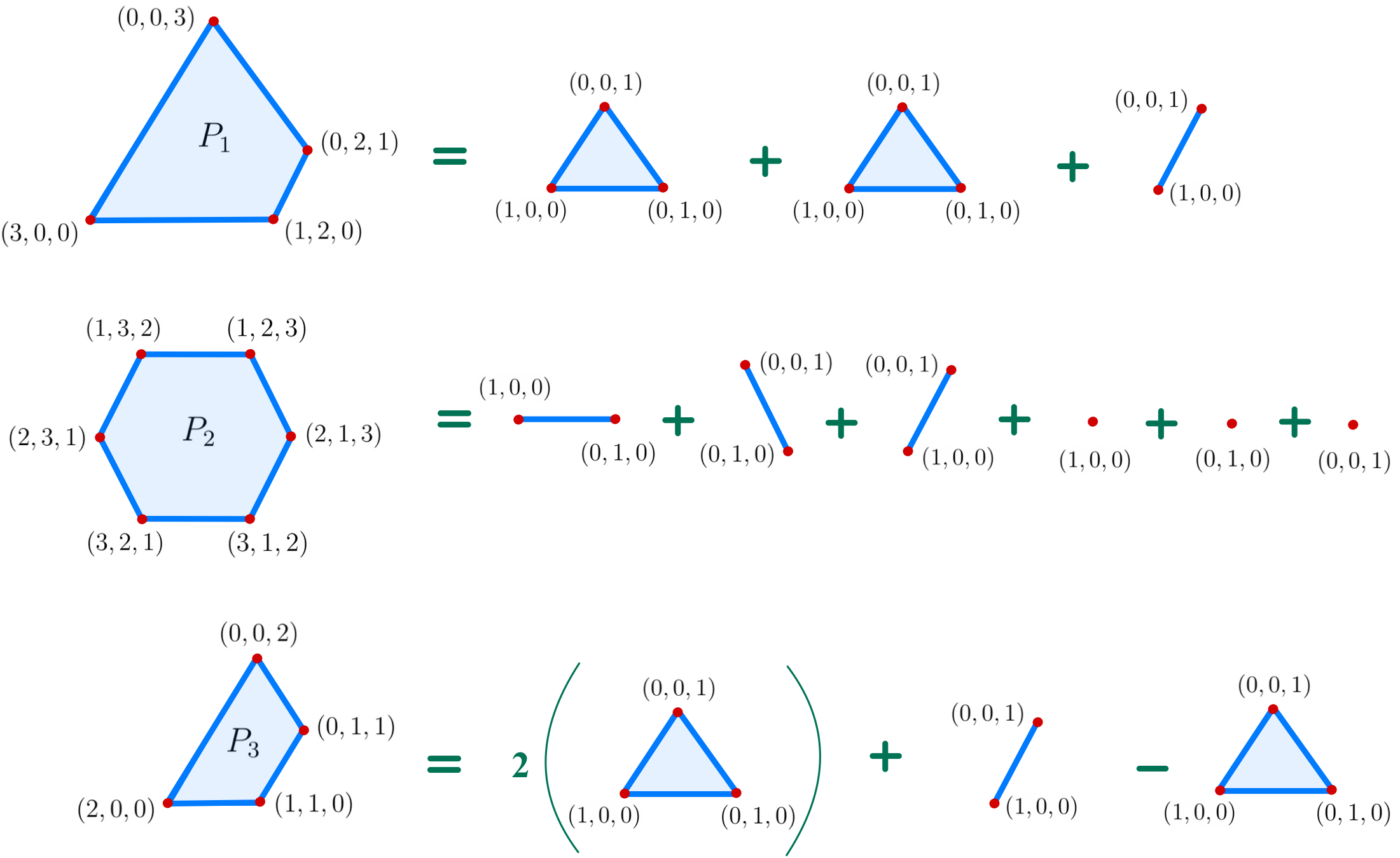}
        \caption{Type-$A$ generalized permutohedra as Minkowski sums of simplices}
        \label{fig: Minkowski-sums-type-a}
    \end{figure}

Thus, we may write every permutohedron $P$ as $P= P(\{y_I\})$. Note that when $P(\{y_I\})$ is integral, the coefficients $y_I$ are integers for all nonempty subsets $I \subseteq [n].$

\begin{ex}\label{ex: type-A-Minkowski-sum} In Figure \ref{fig: Minkowski-sums-type-a}, three examples of type A generalized permutohedra in $\R^3$ are written as Minkowski sums of simplices. One sees that
\begin{align*}
    P_1 &= \Delta_{[3]} + \Delta_{[3]} + \Delta_{\{1,3\}}\\
    P_2 &= \Delta_{\{1,2\}} + \Delta_{\{2,3\}} + \Delta_{\{1,3\}} + \Delta_{\{1\}}+ \Delta_{\{2\}} + \Delta_{\{3\}}\\
    P_3 &= 2\Delta_{[3]} + \Delta_{\{1,3\}} - \Delta_{[3]} = \Delta_{[3]} + \Delta_{\{1,3\}}.
\end{align*}
\end{ex}

Given a bipartite graph $G$ on $m$ left vertices $\{\ell_1, \dots, \ell_m\}$ and $n$ right vertices $\{r_1, \dots, r_n\}$, we denote by $N(v)$ the set of the vertices of $G$ adjacent to the vertex $v$. We call $N(v)$ \emph{the set of neighbors} of $v$. For $ i \in [m]$, we let 
\begin{align}\label{eq: neighbor-of-left-i}
    I_i:= \{j \in [n]\ | \ r_j \in N(\ell_i)\}
\end{align} 
be the set of indices of the neighbors of the left vertex $\ell_i.$ Then, one may define $P_G(y_1, \dots, y_m)$ to be the type A generalized permutohedron $y_1\Delta_{I_{1}} + \cdots + y_m\Delta_{I_{m}}.$ 

Conversely, given $P = y_1\Delta_{I_{1}} + \cdots + y_m\Delta_{I_{m}} \subseteq \R^n$, there is a bipartite graph $G$ on $m$ left and $n$ right vertices such that $P_G(y_1, \dots, y_m) = P.$ Thus, we may interchangeably write $P(\{y_I\})$ as $P_G(y_1, \dots, y_m)$ where $G$ is a corresponding bipartite graph. 

\begin{rem}\label{rem: full-dim-typeA-G-connected}
    If $P_G(y_1, \dots,y_m) \subseteq \R^n$ is $(n-1)$-dimensional, then $G$ is a connected graph.
\end{rem}

\begin{ex}
    Let $G_1, G_2, G_3, G_4$ be bipartite graphs shown in Figure \ref{fig: graphs-type-a}. Then, we can write the type-$A$ generalized permutoheda $P_1, P_2,$ and $P_3$ in Figure \ref{fig: Minkowski-sums-type-a} as
    \begin{align*}
        P_1 &= P_{G_1}(1,1,1)\\
        P_2 &= P_{G_2}(1,1,1,1,1,1)\\
        P_3 &= P_{G_3}(2,1,-1) = P_1 - \Delta_{[3]}= P_{G_4}(1,1,1,-1).
    \end{align*}
\end{ex}

\begin{figure}[h!]
    \centering
    \includegraphics[width=1\linewidth]{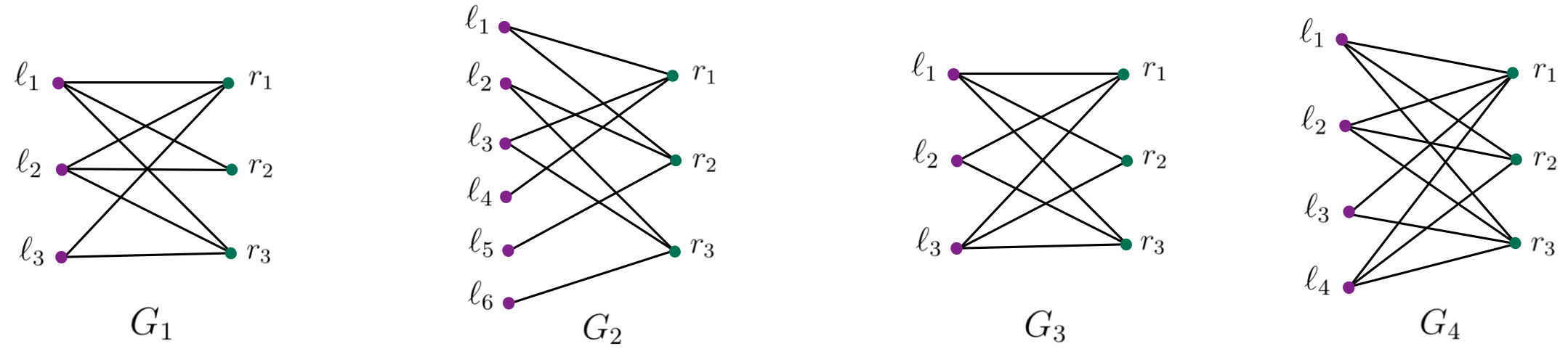}
    \caption{Bipartite graphs associated to polytopes in Figure \ref{fig: Minkowski-sums-type-a}}
    \label{fig: graphs-type-a}
\end{figure}

The rest of this section is devoted to outlining how Postnikov derives a formula for the number of lattice points in type A generalized permutohedra. Since these polytopes are deformations of a permutohedron, it suffices by \cite[Remark 6.4]{Postnikov2005} to only obtain a formula that holds for $P_G(y_1, \dots, y_m) = y_1\Delta_{I_{1}} + \cdots + y_m\Delta_{I_{m}}$ where $y_i$ are positive integers for all $i \in [m]$. Moreover, because every simplex $y\Delta_I$ where $y$ is a positive integer can be written as the Minkowski sum of $y$ copies of the simplex $\Delta_I$, Postnikov first enumerates the number of lattice points in $P_G(1, \dots, 1) = \Delta_{I_{1}} + \cdots + \Delta_{I_{m}}.$ To do so, the Cayley trick is employed to subdivide $P_G(1, \dots, 1)$ into \emph{fine mixed cells} and used as a key to counting the lattice points. We refer the reader to \cite[Section 14]{Postnikov2005} for more details regarding the subdivision.
  
\begin{defn}\label{def: mixed-subdivision}
    Let $P_G(1, \dots, 1) = \Delta_{I_{1}} + \cdots + \Delta_{I_{m}}$. Then, a \emph{fine mixed cell} of $P_G(1, \dots, 1)$ is a polytope $\Pi$ of the form $\Delta_{J_1} + \cdots + \Delta_{J_m}$ where $J_i \subseteq I_i$ for all $i \in [m]$ and satisfies $\dim(\Delta_{J_1}) + \cdots + \dim(\Delta_{J_m}) = \dim(P_G(1, \dots, 1)).$ A \emph{fine mixed subdivision} of $P_G(1,\dots, 1)$ is a polyhedral subdivision of $P_G(1,\dots,1)$ into fine mixed cells.
\end{defn}

\begin{ex}
    Let $P := P_G(1,1,1)$ be the polytope shown in Figure \ref{fig: fineMixedCell-type-a} (also shown as $P_1$ in Figure \ref{fig: Minkowski-sums-type-a}). A fine mixed subdivision of $P$ is drawn inside of $P$. One sees that there are five fine mixed cells, labeled as $\Pi_1, \dots, \Pi_5$ in the subdivision.
\end{ex}

If $\Pi$ is a fine mixed cell of $P_G(1, \dots, 1)$, then there exists a bipartite subgraph $H$ of $G$ such that $\Pi = P_H(1, \dots, 1)$. Moreover, one can show that such a subgraph $H$ is a spanning forest (a subgraph with no cycles) of $G$. We say that $H$ is a \emph{corresponding subgraph} of $\Pi$. For example, when $P := P_G(1,1,1)$ is the polytope shown in Figure \ref{fig: fineMixedCell-type-a}, the fine mixed cells in the given subdivision of $P$, labeled as $\Pi_1, \dots, \Pi_5$, and their corresponding subgraphs $H_1, \dots, H_5$ of $G$ are also shown in the figure. A lattice point in a fine mixed cell of $P_G(1,\dots, 1)$ has a simple form in the following sense.
\begin{lem}[\cite{Postnikov2005}]\label{lem: lattice-points-in-fine-mixed-cell}
    Let $H$ be a bipartite graph on $m$ left vertices and $n$ right vertices. If $H$ is a forest, then every lattice point in $\Pi:=P_{H}(1, \dots, 1)$ is a vertex of $\Pi$. Moreover, every lattice point in $\Pi$ has the form $\bbe_{j_{1}} + \cdots + \bbe_{j_m}$ for some $j_i$ such that $r_{j_i} \in N(\ell_i)$ for all $i \in [m].$ 
\end{lem}

\begin{rem}\label{rem: integers-in-fine-mixed-cells}
    Lemma \ref{lem: lattice-points-in-fine-mixed-cell} states that every lattice point in $\Pi$ corresponds to a transversal $\{r_{j_1}, \dots, r_{j_m}\}$ of the sequence $(N(\ell_1), \dots, N(\ell_m))$ of the neighbors of the left vertices.
\end{rem}

The next two definitions introduce combinatorial objects used for counting the lattice points in type A generalized permutohedra.

\begin{defn}\label{def: G-draconian-type-A}
    Given a bipartite graph $G$ on $m$ left vertices and $n$ right vertices, we let $I_i$ be defined as in equation \eqref{eq: neighbor-of-left-i}. A sequence $(a_1, \dots, a_m)$ of nonnegative integers is a \emph{$G$-draconian} sequence if $a_1 + \cdots a_m = n-1$ and for every nonempty subset $J \subseteq [m]$
    \[\sum_{i \in J}a_i \leq \big| \bigcup_{i \in J} I_i\big| - 1.\]
    The set of all $G$-draconian sequences is denoted by $\mathrm{D}(G).$
\end{defn} 

\begin{defn}\label{def: left-right-degree}
    Given a bipartite graph $G$ on $m$ left vertices $\{\ell_1, \dots, \ell_m\}$ and $n$ right vertices $\{r_1, \dots, r_n\},$ we define the \emph{left degree} $LD(G)$ and the \emph{right degree} $RD(G)$ of $G$, respectively, as
\[LD(G) := (\deg(\ell_1)-1, \dots, \deg(\ell_m)-1) \text{ and } RD(G) := (\deg(r_1)-1, \dots, \deg(r_n)-1).\]
\end{defn}

Fine mixed cells in a fine mixed subdivision $P_G(1,\dots, 1)$ and $G$-draconian sequences are shown in \cite[Lemma 12.6, Lemma 12.8, and Theorem 12.9]{Postnikov2005} to correspond to subgraphs of $G.$

\begin{lem}\label{lem: fine-mixed-subdivision-trees}
    Suppose that $P_G(1,\dots, 1)$ is $(n-1)$-dimensional. Let $\calC = \{\Pi_1, \dots, \Pi_p\}$ be a fine mixed subdivision of $P_G(1, \dots, 1).$ Then, there exists a sequence of spanning trees $H_1, \dots, H_p$ of $G$ satisfying all of the following properties.
    \begin{enumerate}
        \item\label{itm: fine-mixed-cells-trees} For all $i \in [p]$, we have $\Pi_i = P_{H_i}(1, \dots, 1)$
        \item\label{itm: right-deg-lattice-point} The right degrees $RD(H_1),\dots, RD(H_p) \in \Z^n$ are all distinct and
        $$(P_G(1, \dots, 1) - \Delta_{[n]})\cap \Z^n = \{RD(H_1), \dots, RD(H_p)\}.$$
        \item\label{itm: left-deg-g-draconion} The left degrees $LD(H_1),\dots, LD(H_p) \in \Z^m$ are all distinct and 
        $$D(G) = \{LD(H_1), \dots, LD(H_p)\}.$$
    \end{enumerate}
\end{lem}

The following result is an immediate consequence of Lemma \ref{lem: fine-mixed-subdivision-trees}. 

\begin{cor}[\cite{Postnikov2005}]\label{cor: lattice-typeA-Gdraconian-fine-mixed-cell} Suppose that $P_G(1,\dots, 1)$ is $(n-1)$-dimensional. Let $\calC$ be a fine mixed subdivision of $P_G(1,\dots, 1).$ Then, 
\[|(P_G(1, \dots, 1) - \Delta_{[n]})\cap \Z^n| = |\calC| = |D(G)|.\]
\end{cor}

We illustrate the results of both Lemma \ref{lem: fine-mixed-subdivision-trees} and Corollary \ref{cor: lattice-typeA-Gdraconian-fine-mixed-cell} with the polytope shown in Figure \ref{fig: fineMixedCell-type-a} through the following example.

\begin{figure}[h!]
    \centering
    \captionsetup{justification=centering}
    \includegraphics[width=0.92\linewidth]{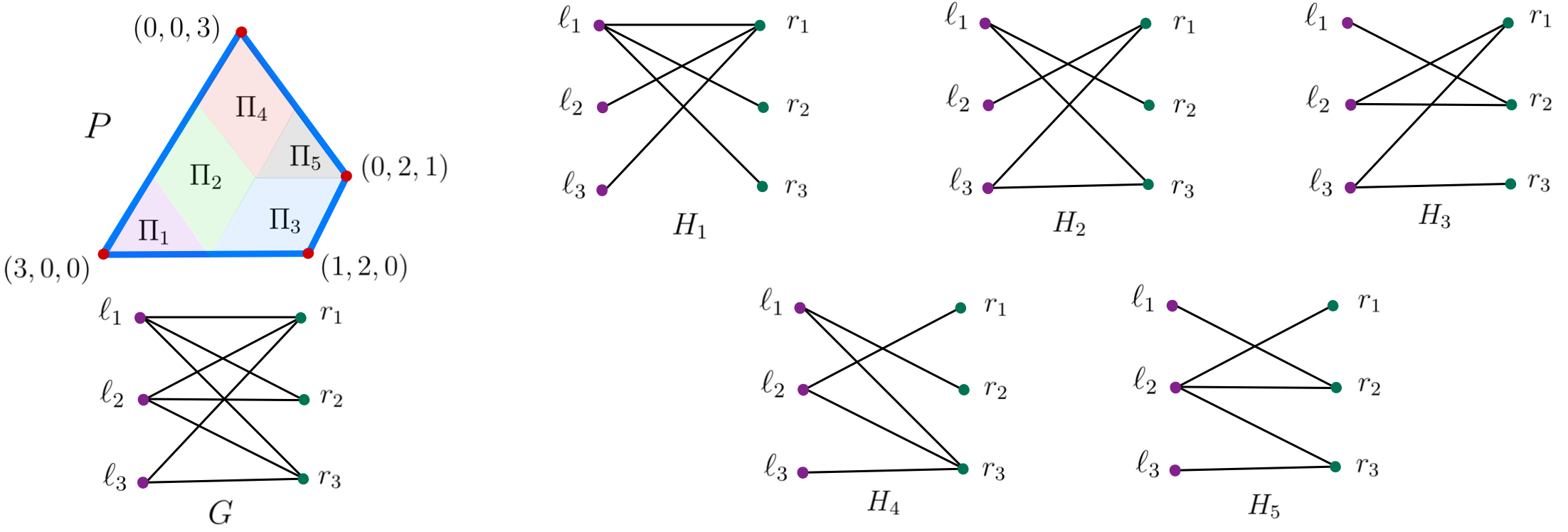}
    \caption{Fine mixed cells in a fine mixed subdivision of $P_G(1,1,1)$ and their corresponding spanning trees}
    \label{fig: fineMixedCell-type-a}
\end{figure}

\begin{ex}\label{ex: Gfine-mixed-subdivision-trees}
    Let $P := P_G(1,1,1)$ be the polytope in Figure \ref{fig: fineMixedCell-type-a}. Since $P - \Delta_{[n]} = P_3$ where $P_3$ is the polytope shown in Figure \ref{fig: Minkowski-sums-type-a}, one has $$(P_G(1,\dots, 1) - \Delta_{[n]})\cap\Z^n = \{(2,0,0), (1,0,1), (1,1,0), (0,0,2), (0,1,1)\}.$$ One can also show that $D(G) =\{(2,0,0), (1,0,1), (0,1,1), (1,1,0), (0,2,0)\}.$ A further direct computation shows that
    \begin{align*}
        (P_G(1,\dots, 1) - \Delta_{[n]})\cap\Z^n &= \{RD(H_1), RD(H_2), RD(H_3), RD(H_4), RD(H_5)\}\\
        D(G) &= \{LD(H_1), LD(H_2), LD(H_3), LD(H_4), LD(H_5)\}.
    \end{align*}
    Moreover, by letting $\calC$ be the fine mixed subdivision of $P$ shown in Figure \ref{fig: fineMixedCell-type-a}, one sees that $|(P_G(1,1, 1) - \Delta_{[n]})\cap\Z^n| = |\calC|= |D(G)| = 5.$
\end{ex}

Corollary \ref{cor: lattice-typeA-Gdraconian-fine-mixed-cell} implies that the number of lattice points in $P_G-\Delta_{[n]}$ is given by
\begin{align}\label{eq: lattice-points-Gdraconians}
    |(P_G-\Delta_{[n]})\cap \Z^n| = \sum_{\bba \in D(G)}1 = \sum_{\bba \in D(G)}\binom{1+ a_1 -1}{a_1}\cdots\binom{1+ a_m -1}{a_m}
\end{align}
Postnikov obtains the following key lemma by expressing $P_G(y_1, \dots, y_m)$ as
\[P_{G'}(1,\dots,1) = \underbrace{\Delta_{I_1} + \cdots +\Delta_{I_1}}_{y_1 \text{ terms }} + \cdots + \underbrace{\Delta_{I_m} + \cdots + \Delta_{I_m}}_{y_m \text{ terms }},\]
for an appropriate $G'$, and applying a simple binomial identity to the right-hand side of  \eqref{eq: lattice-points-Gdraconians}.

\begin{lem}\label{lem: lattice-type-A} Suppose that $P_G(y_1, \dots, y_m) = y_1\Delta_{I_{1}} + \cdots + y_m\Delta_{I_{m}}$ where $y_i$ are integers for all $i \in [m]$. Then, the number of lattice points in $P_G(y_1, \dots, y_m)-\Delta_{[n]}$ is given by
\begin{align}\label{eq: lattice-type-A}
    |(P_G(y_1, \dots, y_m)-\Delta_{[n]})\cap \Z^n| = \sum_{\bba \in D(G)}\binom{y_1+ a_1 -1}{a_1}\cdots\binom{y_m+ a_m -1}{a_m}.
\end{align}
\end{lem}

\subsection{Type B generalized Permutohedra}

A \emph{type B generalized permutohedron} in $\R^n$ is a polytope whose edges are parallel to $\bbe_i+\bbe_j, \bbe_i-\bbe_j,$ or $\bbe_i$ for some $i,j \in [n]$ where $\bbe_i$ denotes the standard basis vector in $\R^n$. The name ``type B'' comes from the fact that the edge directions of these polytopes are parallel to some vectors in  type B positive root system $\{\bbe_1, \dots, \bbe_n, \bbe_i\pm\bbe_j  \ | \ 1 \leq i < j \leq n\}$. We note that these polytopes are also known as \emph{generalized signed permutohedra} (see \cite{Euretall2024}). By definition, every type A generalized permutohedron is a type B generalized permutohedron. Other notable examples in the family include parking function polytopes \cite{Amanbayeva_2021}\cite{LiuThawinrak2025} and polymatroids \cite{edmonds1970}\cite[Theorem 17.1]{ Fujishige2005submodular}.

\begin{figure}[h!]
    \centering
    \includegraphics[width=1\linewidth]{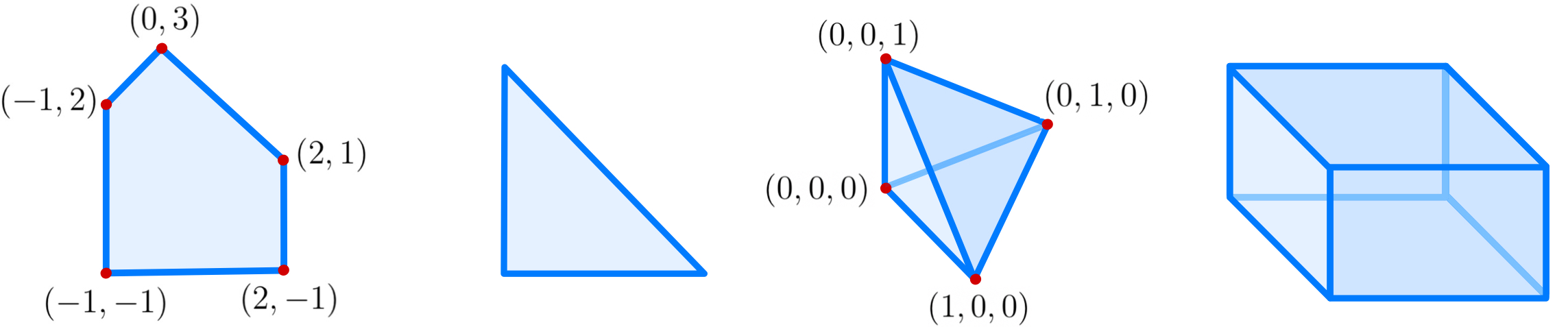}
    \caption{Examples of type B generalized permutohedra in $\R^2$ and $\R^3$}
    \label{fig: type-b}
\end{figure}

One may equivalently define a type B generalized permutohedron in $\R^n$ as a polytope in which its normal fan coarsens the fan whose full dimensional cones are defined by the chambers of the arrangement of hyperplanes
\begin{align}\label{eq: B-n-fan}\begin{cases}
    H^+_{i,j} &:= \{ (c_1, \dots, c_n) \in \R^n \ | \ c_i + c_j =0\} \text{ for all } i\neq j \in [n],\\
    H^-_{i,j} &:= \{ (c_1, \dots, c_n) \in \R^n \ | \ c_i - c_j =0\} \text{ for all } i\neq j \in [n], \text{ and}\\
    H_i &:= \{ (c_1, \dots, c_n) \in \R^n \ | \ c_i =0\} \text{ for all } i \in [n].
\end{cases}
\end{align}
The arrangement of hyperplanes in equation \eqref{eq: B-n-fan} is called the \emph{type B Coxeter arrangement}. The fan whose full dimensional cones are defined by the chambers of the type B Coxeter arrangement is known as the \emph{$B_n$ permutohedral fan} and is denoted by $\Sigma_{B_n}.$ We refer the reader to \cite{Bastidas2021} and \cite{Euretall2024} for more details regarding type B generalized permutohedra beyond what are presented here.
    
Let $[\Bar{n}]:= \{\Bar{1}, \dots, \Bar{n}\}$ and $\nn := [n] \sqcup [\Bar{n}] = \{1, \dots, n, \Bar{1}, \dots, \Bar{n}\}.$ For $i \in [n]$, we define $\bbe_{\Bar{i}} := -\bbe_i$ where $\bbe_i$ is a standard basis vector of $\R^n.$ 

\begin{defn}
The set $\ads$ of \emph{admissible subsets} of $\nn$ is defined to be 
\[\ads := \{S \subset \nn \,|\, S\neq \emptyset \text{ and } \{i, \Bar{i}\} \not\subset S \text{ for all } i\} \text{ and } \ads_n := \{S \in \ads\,|\, |S| = n\}.\]
\end{defn} 

For $T \in \ads_n$, we define $\R_T := \{\bbx \in \R^n \ | \ \bbx\cdot e_i \geq 0 \text{ for all } i \in T\}$ to be the associated orthant to $T$ of $\R^n.$ For instance, when $T = \{1,\dots, n\}$, we obtain $\R_T = \R^n_{\geq 0}$ the first orthant, whereas when $T = \{\Bar{1},\dots, \Bar{n}\}$, we obtain the opposite orthant $\R_T = \R^n_{\leq 0}$.

\begin{defn}
    Let $S \in \ads$ and $T \in \ads_n$. We denote by
    \begin{enumerate}
        \item $\sym_S$ the simplex $\conv(\mathbf{0},\bbe_i\ |\ i \in S)$.
        \item $\square_T$ the unit cube $\sum_{i \in T}\sym_{\{i\}}$ in the orthant $\R_T.$
    \end{enumerate}

\end{defn}

Bastidas showed from a study of Tits algebras in \cite{Bastidas2021} that every type B generalized permutohedron can be written as a Minkowski sum (and difference) of the simplices $\Delta_S^0$. This result was also established later in \cite{Euretall2024} by Eur, Fink, Larson, and Spink from a study of delta-matroids.

\begin{lem}[\cite{Bastidas2021}, \cite{Euretall2024}]\label{lem: type-B-Minkowski-sum}
    For every type B generalized permutohedron $P \subset \R^n$, there exists a partition of $\ads$ into two disjoint subsets $M_1$ and $M_1$ such that
\begin{align}\label{eq: typeB-Minkowski-sum}
    P + \sum_{S \in M_1}y_S\Delta^0_S = \sum_{S \in M_2}y_S\Delta^0_S
\end{align}
for some $y_S > 0$ for all $I \in M_1$ and some $y_S \geq 0$ for all $I \in M_2.$ Moreover, the coefficients $y_S$ in \eqref{eq: typeB-Minkowski-sum} are unique. In particular, there exists a sequence of real numbers $(y_S)_{S \in \ads}$ such that 
\begin{align}\label{eq: typeB-Minkowski-sum-2}
    P = \sum_{S \in \ads} y_S\Delta^0_S.
\end{align}
\end{lem}

One sees that Lemma \ref{lem: type-B-Minkowski-sum} generalizes to type B the type A result by Postnikov in Lemma \ref{lem: type-A-Minkowski-sum}. Thus, one may write every type B generalized permutohedron as $P(\{y_S\})$. If $P(\{y_S\})$ is integral, then it can be shown that the coefficients $y_S$ are integers for all $S \in \ads$.

It is also important to note that whenever we write a type $B$ generalized permutohedron as in \eqref{eq: typeB-Minkowski-sum-2}, the coefficients $y_S$ are the unique real numbers given in \eqref{eq: typeB-Minkowski-sum}.

\begin{ex}\label{ex: Minkowski-sums-type-b}
    In Figure \ref{fig: Minkowski-sums-type-b}, the type B generalized permutohedra $P_1$ and $P_2$ are written as $P_1 = \sym_{[2]} + \sym_{[2]} + \sym_{\{\Bar{1},2\}} + \sym_{\{\Bar{2}\}} \text{ and } P_2 = P_1 -  \square_{[2]}= 2\sym_{[2]} +\sym_{\{\Bar{1},2\}} + \sym_{\{\Bar{2}\}} - \sym_{\{1\}} - \sym_{\{2\}}.$
\end{ex}

\begin{figure}
    \centering
    \includegraphics[width=1\linewidth]{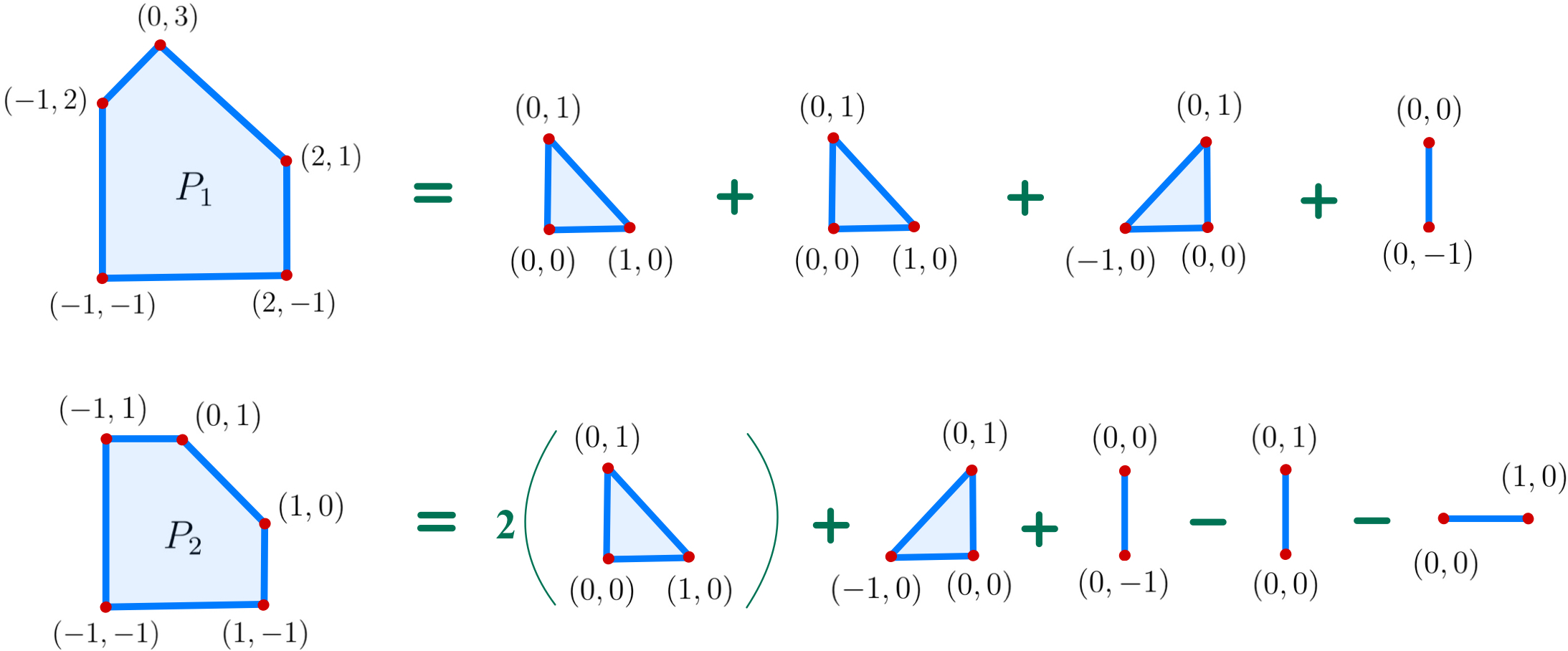}
    \caption{Type-$B$ generalized permutohedra as Minkowski sums of simplices}
    \label{fig: Minkowski-sums-type-b}
\end{figure}

\begin{defn}\label{def: sign-transversals} Let $S_1, \dots, S_n$ be admissible subsets of $[n,\Bar{n}].$ The \emph{signed transversal} of $(S_1, \dots, S_n)$ is an admissible subset $T \in \ads_n$ such that there exists a bijection $g:[n] \longrightarrow T$ satisfying $g(i) \in S_i$ for all $i \in [n].$
\end{defn}

In \cite{Euretall2024}, Eur, Fink, Larson, and Spink give formulas for the volume and the number of lattice points of type B generalized permutohedra in terms of signed transversals as follows.

\begin{lem}[\cite{Euretall2024}]\label{lem: typeB-Vol-Ehrhart} Suppose that $P(\{y_S\}) = \sum_{S \in \ads} y_S\Delta^0_S$ where $y_S \in \Z$ for all $S \in \ads$. Then, 
    \begin{align}
    \nvol(P(\{y_S\})) &= \sum_{(S_1, \dots, S_n)}|\text{signed transversals of } (S_1, \dots, S_n)|\cdot y_{S_1}\cdots y_{S_n}, \text{ and }\nonumber\\
    \label{eq: typeB-Ehrhart-Eur}
        |(P(\{y_S\})-\square_{[n]})\cap \Z^n| &= \sum_{(S_1, \dots, S_n)}|\text{signed transversals of } (S_1, \dots, S_n)|\cdot \Psi\left(y_{S_1}\cdots y_{S_n} \right) 
    \end{align}
    where $\Psi$ is the linear operator on the set of polynomials that maps each monomial $x^{a_1}_1\cdots x^{a_m}_m$ to $\frac{a_1!\cdots a_m!}{(a_1 + \cdots + a_m)!}\binom{x_1}{a_1}\cdots \binom{x_m}{a_m}$.
\end{lem}

\section{Thinking of B from A}

We begin by outlining some key properties of certain type $B$ generalized permutohedra and then demonstrate how they can be realized as type $A$ generalized permutohedra.

The set of vertices of $\sym_S$ is $\{\bbe_j\ | \ j \in S\}\cup \{\mathbf{0}\}$. Thus, one can express every point $\bbx \in \sym_S$ as
    \begin{equation}\label{eq: point_in_Delta_S}
        \bbx = \sum_{j \in S}\lambda_j \bbe_j \text{ for some real number }  \lambda_j \geq 0 \text{ such that } \sum_{j \in S}\lambda_j \leq 1.
    \end{equation}

\begin{rem}\label{rem: separation_by_T}
    For $S, T \in \ads,$ let $\bbx \in \sym_S$.  It is easy to see from equation \eqref{eq: point_in_Delta_S} that we can write $\bbx = \bba + \bbb$ for some $\bba \in \sym_{S\cap T}$ and $\bbb \in \sym_{S\backslash T}.$ Thus, if $\bby \in P = y_1\Delta^0_{S_1} + \cdots + y_m\Delta^0_{S_m}$ where $y_i > 0$ for all $i \in [m]$, then we have $\bby = \bbu + \bbv$ where $\bbu \in y_1\Delta^0_{S_1\cap T} + \cdots + y_m\Delta^0_{S_m\cap T}$ and $\bbv \in y_1\Delta^0_{S_1\backslash T} + \cdots + y_m\Delta^0_{S_m \backslash T}$.
\end{rem}

\begin{lem}\label{lem: inP}
    Let $P = y_1\Delta^0_{S_1} + \cdots + y_m\Delta^0_{S_m}$ where $y_i$ are positive real numbers and $S_i \in \ads$ for all $i \in [m]$. If the point $(x_1, \dots, x_n)$ lies in $P$, then, for every $(r_1, \dots, r_n) \in \R^n$ such that $0 \leq r_j \leq 1$ for all $j \in [n]$, the point $(r_1x_1, \dots, r_nx_n)$ also lies in $P$.
\end{lem}

\begin{proof}
    The statement holds for points in the simplices $y\Delta^0_{S}$ for all $y > 0$ and $S \in \ads$, since $\Delta^0_{S}$ is the set of all points $\bbx$ given in \eqref{eq: point_in_Delta_S}.

    Now consider $\bbx = (x_1, \dots, x_n) \in P.$ Then, $\bbx = \bbz_1 + \cdots + \bbz_m$ for some $\bbz_i \in y_i\sym_{S_i},$ $i\in [m]$. Let $\bbr = (r_1, \dots, r_n) \in \R^n$ satisfy $0 \leq r_j \leq 1$ for all $j \in [n]$. Then, the points $\bbr\cdot\bbz_i$ lie in $y_i\sym_{S_i}$ for all $i \in [m]$. Thus, the point $(r_1x_1, \dots, r_n x_n) = \bbr\cdot \bbx = \bbr\cdot\bbz_1 + \cdots + \bbr\cdot \bbz_m$
    also lies in $P.$
\end{proof}

\begin{defn}
    Let $P = y_1\Delta^0_{S_1} + \cdots + y_m\Delta^0_{S_m}$ where $y_i$ are positive real numbers and $S_i \in \ads$ for all $i \in [m]$. For $T \in \ads_n,$ we define $P_T := y_1\Delta^0_{S_1\cap T} + \cdots + y_m\Delta^0_{S_m\cap T}.$
\end{defn}

\begin{lem}\label{lem: P_T}
    Let $P = y_1\Delta^0_{S_1} + \cdots + y_m\Delta^0_{S_m}$ where $y_i$ are positive real numbers and $S_i \in \ads$ for all $i \in [m]$. Then, for $T \in \ads_n,$ we have $P_T = P \cap \R_T.$ That is, $P_T$ equals the polytope $P$ intersecting with the orthant associated to $T.$
\end{lem}

\begin{proof}
     Clearly, $P_T \subseteq P$. Since $y_i\Delta^0_{S_i\cap T} \subset \R_T$ for all $i \in [m],$ it follows that $P_T \subseteq \R_T$. Thus, $P_T \subseteq P \cap \R^n_T.$

    Next, we show that $P \cap \R_T \subseteq P_T.$ Due to symmetry, it suffices to only show this for $T = [n].$ That is, we only need to consider $P \cap \R_T = P \cap \R^n_{\geq 0}$. Let $\bbx = (x_1, \dots, x_n) \in P \cap \R^n_{\geq 0}.$ Then, by Remark \ref{rem: separation_by_T}, $\bbx = \bba + \bbb$ where $\bba = (a_1, \dots, a_n)  \in P_T$ and $\bbb = (b_1, \dots, b_n) \in y_1\Delta^0_{S_1\backslash T} + \cdots + y_m\Delta^0_{S_m \backslash T} = P_{T^c}$ where $T^c := [n, \Bar{n}]\backslash T = \{\Bar{1}, \dots, \Bar{n}\}.$ This implies  $a_i \geq 0$ and $b_i \leq 0$ for all $i \in [n]$. Since $\bbx = \bba + \bbb \in \R^n_{\geq 0}$, it follows that $0\leq x_i = a_i + b_i \leq a_i$ for all $i \in [n]$. Thus, $\bbx = (r_1a_1, \dots, r_na_n)$ for some $(r_1, \dots, r_n)$ with $0 \leq r_i \leq 1$ for all $i\in [n]$. Therefore, by Lemma \ref{lem: inP}, we have $\bbx \in P_T.$ This shows $P \cap \R_T \subseteq P_T$ as desired. 
\end{proof}

The result of Lemma \ref{lem: P_T} is illustrated through the following example.

\begin{ex}\label{ex: P_T}
    Let $P = \sym_{[2]} + \sym_{[2]} + \sym_{\{\Bar{1},2\}} + \sym_{\{\Bar{2}\}} \subset \R^2$ be the type B generalized permutohedron shown in Figure \ref{fig: P_T} (also shown as $P_1$ in Figure \ref{fig: type-b}). Then, $P_T$ for $T \in \ads_2$ are also depicted in the same figure. One has that 
    \begin{align*}
        P_{\{1,2\}} &= \sym_{[2]} + \sym_{[2]} + \sym_{\{2\}}&
        P_{\{1,\Bar{2}\}} &= \sym_{\{1\}} + \sym_{\{1\}} + \sym_{\{\Bar{2}\}}\\
        P_{\{\Bar{1},2\}} &=\sym_{\{2\}} + \sym_{\{2\}} + \sym_{\{\Bar{1},2\}}&
        P_{\{\Bar{1},\Bar{2}\}} &=\sym_{\{\Bar{1}\}} + \sym_{\{\Bar{2}\}}.
    \end{align*}
\end{ex}

\begin{figure}
    \centering
    \captionsetup{justification=centering}
    \includegraphics[width=0.9\linewidth]{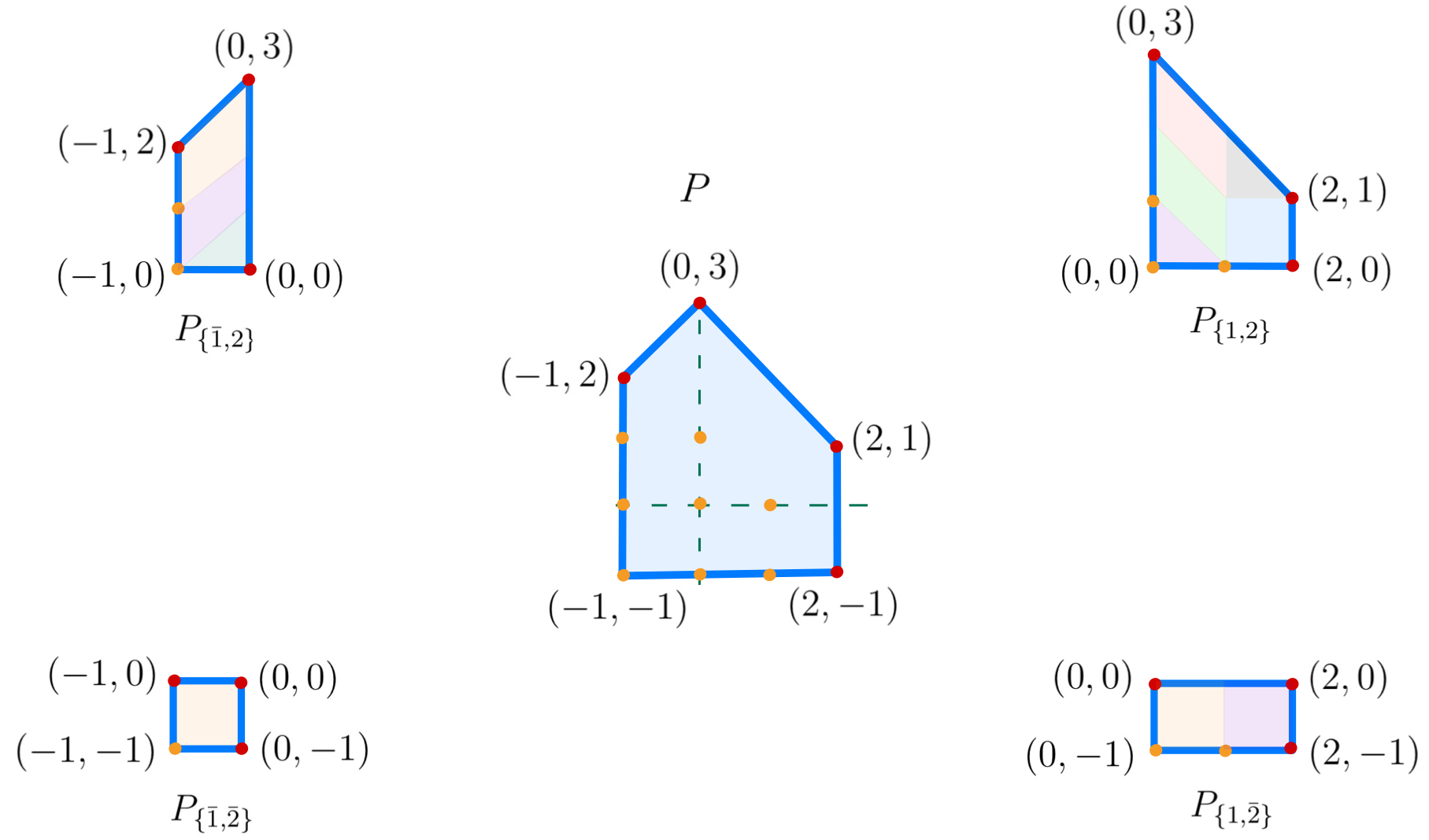}
    \caption{$P$ and $P_T$, the lattice points in $P-\square_{[n]}$ and $P_T-\square_{[n]}$ (the orange points), and fine mixed subdivisions of $P_T$ for $T \in \ads_2$}
    \label{fig: P_T}
\end{figure}

\begin{defn} For $i \in \N$, we define $||i||=||\Bar{i}|| := i,$ and, for $S \in \ads,$ $||S||:= \{||j|| \ | \ j \in S\}$.
\end{defn}

For $L \subseteq [0,n]$, let us denote by $\R^{L}$ the space of all points $(x_\ell)_{\ell \in L}$ where $x_\ell \in \R$ for all $\ell \in L.$

\begin{lem}\label{lem: biject-to-type-A}
    Let $T \in \ads_n$ and $P = y_1\Delta^0_{S_1} + \cdots + y_m\Delta^0_{S_m}$ where $y_i$ are positive integers and $S_i \in \ads$ for all $i \in [m]$. Then, $P_T$ is integrally equivalent to the type A generalized permutohedron
    $Q_T := y_1\Delta_{I_1} + \cdots + y_m\Delta_{I_m} \subseteq \R^{[0,n]}$ where $I_i := \{0\}\cup ||S_i\cap T||$ for all $i \in [m]$.
\end{lem}

\begin{proof}
Let $\varphi_T: \R^{[0,n]}\longrightarrow \R^n$ be the linear map from $\R^{[0,n]}$ onto $\R^n$ defined by 
    \begin{align}
        \label{eq: projection} \varphi_T(x_0, x_1, \dots, x_n) = \sum_{j \in T}x_{||j||} \bbe_j.
    \end{align}
That is, $\varphi_T$ is the projection that projects the first orthant of $\R^{[0,n]}$ onto the orthant $\R_T$ of $\R^n$. Note that, for $S \in \ads$, the simplex $y\Delta_{\{0\}\cup||S||}$ is given by the set of all points $\bbx \in \R^{[0,n]}$ such that $x_i = 0$ for all $i \in [n]\backslash ||S||$ and
\[\sum_{i \in \{0\}\cup ||S||}x_i = y, 0 \leq x_i \leq y \text{ for all } i \in \{0\}\cup ||S||,\] 
and that
\[\mathrm{aff}\left(y\Delta_{\{0\}\cup||S||}\right)= \left\{\bbx \in \R^{[0,n]}\mid \sum_{i \in \{0\}\cup ||S||}x_i = y \text{ and } x_i = 0 \text{ for } i \in [n]\backslash ||S||\right\}.\] 
Thus, it is easy to see that $\varphi_T$ defines an invertible map from $\mathrm{aff}(y\Delta_{\{0\}\cup||S||})$ to $\mathrm{aff}(y\Delta^0_{S}) = \{\bbx \in \R^n \ | \ x_i = 0 \text{ for } i \in [n]\backslash ||S||\}$ and preserves the lattices points between $y\Delta_{\{0\}\cup||S||}$ and $y\Delta^0_{S}$. Additionally, the map $\varphi_T$ defines an invertible transformation from the affine space
\[V= \left\{\bbx \in \R^{[0,n]}\mid \sum_{i \in \{0\}\cup ||T||}x_i = y_1 + \cdots + y_m \text{ and } x_i = 0 \text{ for } i \in [n]\backslash ||T||\right\}\] 
to $\R^n$ and preserves their lattice points. Since $\mathrm{aff}(Q_T) = \mathrm{aff}(y_1\Delta_{I_1})+\cdots + \mathrm{aff}(y_m\Delta_{I_m}) \subseteq V$ and $\mathrm{aff}(P_T) \subseteq \R^n$, it follows that $\varphi_T$ also defines an injective map from $\mathrm{aff}(Q_T)$ to $\mathrm{aff}(P_T)$. By linearity, one further sees that $\varphi_T$ defines a surjective map from $\mathrm{aff}(Q_T)$ to $\mathrm{aff}(P_T)$ and preserves the lattice points between $Q_T$ and $P_T$. Therefore, $Q_T$ and $P_T$ are integrally equivalent.
\end{proof}

  Informally, Lemma \ref{lem: biject-to-type-A} states that we can view the intersection of $P$ and the orthant $\R_T$ as a type A generalized permutohedron. This is where we can apply some of the techniques and tools introduced by Postnikov in \cite{Postnikov2005} to type B generalized permutohedra. In particular, this realization allows us to define fine mixed subdivisions and associate a bipartite graph to $P_T$. 
  
  \begin{defn}\label{def: bipartite-for-typeB}
      With the same assumptions as given in Lemma \ref{lem: biject-to-type-A}, we define the corresponding bipartite graph $G(P_T)$ of $P_T$ to be the bipartite graph $G(Q_T)$ corresponding to the type A permutohedron $Q_T$ given Lemma \ref{lem: biject-to-type-A}.
  \end{defn}

  \begin{ex}
      Let $P = \sym_{[2]} + \sym_{[2]} + \sym_{\{\Bar{1},2\}} + \sym_{\{\Bar{2}\}} \subset \R^2$ be the type B generalized permutohedron shown in Figure \ref{fig: P_T}. One sees from Figure \ref{fig: IntegrallyEquiv} that $P_{\{1,2\}}$ is integrally equivalent to the type $A$ generalized permutohedron $Q_{\{1,2\}} \subset \R^{[0,2]}$ (also shown as $P$ in Figure \ref{fig: fineMixedCell-type-a}) via the projection $(x_0,x_1,x_2) \mapsto (x_1, x_2)$. Moreover, the associated bipartite graph $G(P_{\{1,2\}})$ of $P_{\{1,2\}}$ (shown in Figure \ref{fig: IntegrallyEquiv}) is the graph $G(Q_{\{1,2\}}) = G$ where $G$ is shown in Figure \ref{fig: fineMixedCell-type-a}, except that the right vertices $r_1, r_2, r_3$ are relabeled respectively as $r_0,r_1,r_2$. The fine mixed cells in $P_{\{1,2\}}$ are also obtained by projecting the fine mixed cells in $Q_{\{1,2\}}$.  
  \end{ex}

  \begin{figure}[h!]
      \centering
      \includegraphics[width=0.85\linewidth]{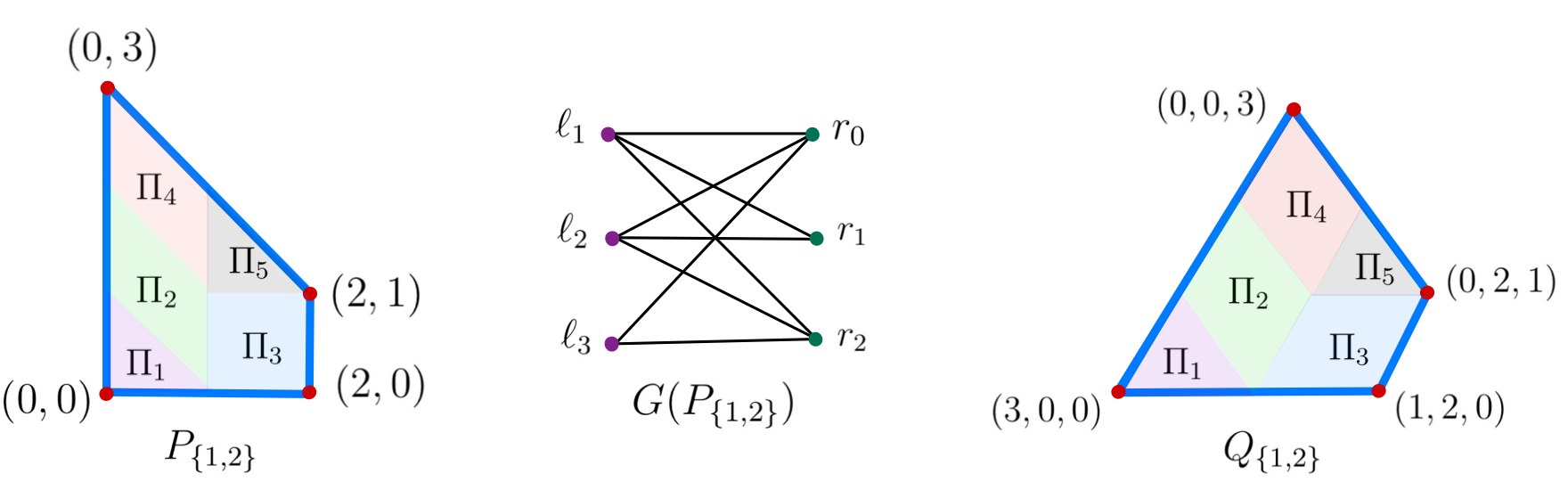}
      \caption{$P_{\{1,2\}} \subset \R^2$ and $Q_{\{1,2\}} \subset \R^{[0,2]}$ are integrally equivalent}
      \label{fig: IntegrallyEquiv}
  \end{figure}

It is important to note that when we view $P_T$ as a type A generalized permutohedron, we always set $\bbe_0$ to be the zero vector.

\section{Counting Lattice Points from A to B}

In similar fashion to how Postnikov enumerates the lattice points in type A generalized permutohedra, we only need to derive a formula for the number of lattice points in type B generalized permutohedra $P = \sum_{S \in \ads} y_S\Delta^0_S$ that holds for nonnegative integers $y_S$, for all $S \in \ads.$ Once such a formula is obtained, it will extend to hold for all integers $y_S.$ Since  $y\sym_S$ can be written as Minkowski sum of $y$ copies of $\sym_S,$ we first attempt to enumerate the number of lattice points in $P = \Delta^0_{S_1} + \cdots + \Delta^0_{S_m}$ where $S_i \in \ads.$

\begin{defn}\label{def: G-draconian-seq-for-typeB}
    Let $T \in \ads_n$ and $P = y_1\Delta^0_{S_1} + \cdots + y_m\Delta^0_{S_m}$ where $y_i$ are nonzero integers and $S_i \in \ads$ for all $i \in [m]$. A \emph{$G$-draconian} sequence of $P_T$ is a sequence $(a_1, \dots, a_m)$ of nonnegative integers satisfying $a_1 + \cdots + a_m = n$ and for every nonempty subset $I \subseteq [m]$
    \[\sum_{i \in I}a_i \leq \big| \bigcup_{i \in I} S_i\cap T \big|.\]
    We denote by $\mathrm{D}(P_T)$ the set of all $G$-draconian sequences, and denote by $\mathrm{D}(P_T)\cap \square_{[m]}$ the set of $G$-draconian sequences in which $a_i \leq 1$ for all $i \in [m]$.
\end{defn}

One observes that the definition of $G$-draconian sequence in Definition \ref{def: G-draconian-seq-for-typeB} is precisely the definition of $G(P_T)$-draconian sequence in Definition \ref{def: G-draconian-type-A}. One also sees that if $P_T$ is not $n$-dimensional, $G(P_T)$ must have a right vertex of degree zero. Consequently, $a_1 + \cdots + a_m \leq \big| \bigcup_{i \in [m]} S_i\cap T \big| < n.$ Thus, $\mathrm{D}(P_T) = \emptyset$ provided $P_T$ is not $n$-dimensional.

Recall that, for certain type A generalized permutohedra, the lattice points are shown to correspond bijectively to both their $G$-draconian sequences and their fine mixed cells, thereby yielding formula \eqref{eq: lattice-type-A} in Lemma \ref{lem: lattice-type-A}. Employing Postnikov's approach, we establish in the next lemma an analogous correspondence where the lattice points correspond bijectively to both the $G$-draconian sequences in which $a_i \leq 1$ and those fine mixed cells that are zonotopes, i.e., those cells that are Minkowski sums of line segments.

\begin{lem}\label{lem: fine-mixed-subdivision-latticepts}
    Suppose that $P = \Delta^0_{S_1} + \cdots + \Delta^0_{S_m}$ where $S_i \in \ads$ for all $i \in [m]$. Let $\calC^*$ be the set of those fine mixed cells in a fine mixed subdivision of $P$ that are zonotopes. If $T \in \ads_n$ is an admissible subset such that $P_T$ is $n$-dimensional, then 
    \[|(P_T - \square_T)\cap \Z^n| = |\mathrm{D}(P_T)\cap \square_{[m]}|= |\calC^*|.\]
\end{lem}

We illustrate the lemma’s result through the following example, after which we develop the necessary auxiliary results before presenting a proof.

\begin{ex}
    Figure \ref{fig: P_T} shows $P = \sym_{[2]} + \sym_{[2]} + \sym_{\{\Bar{1},2\}} + \sym_{\{\Bar{2}\}} \subset \R^2$ together with $P_T$ for $T \in \ads_2$, and the lattice points in $P - \square_{[2]}$ and $P_T - \square_{[2]}$ (the orange points). One sees that, for every $T \in \ads_2$, the number of lattice points in $P_T -\square_{[2]}$ equals the number of fine mixed cells in a subdivision of $P_T$ that are zonotopes. Moreover, one has 
    \[|(P - \square_{[2]})\cap \Z^2| = \sum_{T \in \ads_2}|(P_T - \square_{[2]})\cap \Z^2| = 8.\]
\end{ex}

\begin{defn}\label{def: no-translation-cell} Let $T \in \ads_n$, and $P = \Delta^0_{S_1} + \cdots + \Delta^0_{S_m}$ where $S_i \in \ads$ for all $i \in [m]$. Let $\Pi$ be a fine mixed cell of $P_T$, and $H$ be the bipartite subgraph of $G(P_T)$ corresponding to $\Pi,$ i.e., $\Pi$ is integrally equivalent to the type A generalized permutohedron $P_{{H}}(1,\dots, 1) \subset \R^{[0,n]}.$ Suppose that $\Pi = \Delta_{I_1} + \cdots + \Delta_{I_m}$  where  $I_i \subseteq \{0\}\cup S_i\cap T$ for all $i \in [m]$. Let $K = \{I_i\ | \ i \in [m] \text{ and } |I_i| \geq 2\}$. Then, we define $\hat{\Pi}$ to be the polytope 
    \[\hat{\Pi} := \sum_{I \in K}\Delta_I\]
obtained by removing the translating factor from $\Pi,$ and denote by $\hat{H}$ the induced bipartite subgraph of $H$ corresponding to $\hat{\Pi}.$
\end{defn}

We highlight some of the basic properties of $\hat{\Pi}$ in the following two remarks.

\begin{rem}\label{rem: translation-of-Pi}
    Let $K' = \{I_i \ | \ i \in [m] \text{ and } |I_i| = 1\}$. The fine mixed cell $\Pi$ is the translation $\bbx + \hat{\Pi}$ of $\hat{\Pi}$ where $\bbx$ is the integral vector (translating factor) $$\bbx = \sum_{I \in K'}\Delta_I.$$ 
\end{rem}

\begin{rem}\label{rem: at-most-n-left-vertices}
    If we assume further that $P_T$ is $n$-dimensional ($G(P_T)$ is connected), then $H$ is a spanning tree of $G(P_T)$ on $m$ left and $n+1$ right vertices. Consequently, we have that $\hat{H}$ is a bipartite graph on $n+1$ right vertices in which every left vertex has degree at least two. Moreover, $\hat{H}$ has at most $n$ left vertices. The tree $\hat{H}$ has exactly $n$ left vertices if and only if every left vertex of $\hat{H}$ has degree two.
\end{rem} 

\begin{lem}\label{lem: no-common-interior} Let $P = \Delta^0_{S_1} + \cdots + \Delta^0_{S_m}$ where $S_i \in \ads$ for all $i \in [m]$. Suppose that $T \in \ads_n$ is an admissible subset in which $P_T$ is $n$-dimensional. If $\Pi$ is a fine mixed cell of $P_T$, then $\Pi$ and $\bby + \Pi$ have no common interior for all vectors $\bby \in \Z^n\backslash\{\mathbf{0}\}.$
\end{lem} 

\begin{proof}
    Suppose that $\Pi$ is a fine mixed cell of the form $\Pi = \Delta_{I_1} + \cdots + \Delta_{I_m}$ where  $I_i \subseteq \{0\}\cup S_i\cap T$ for all $i \in [m]$. We claim that $\Pi$ lies in an integral translation of a \emph{fundamental parallelepiped}. It is well-known that any fundamental parallelepiped has no common interior with its integral translation. Thus, the conclusion of the lemma follows from this claim. Hence, it only remains to prove the claim.
    
    Let $K = \{I_i\ | \ i \in [m] \text{ and } |I_i| \geq 2\}$. Then, 
    \[\hat{\Pi} := \sum_{I \in K}\Delta_I.\]
    For each $I \in K$, let us write $I= \{j_1, \dots, j_{|I|}\ | \ |j_i| < |j_{i+1}|, \text{ for all } 1 \leq i < |I| \}$, and define the corresponding set of intervals 
    $\binom{I}{2}^* := \{[\bbe_{j_1}, \bbe_{j_2}], [\mathbf{0}, \bbe_{j_3} - \bbe_{j_1}], \dots, [\mathbf{0},\bbe_{j_{|I|}} - \bbe_{j_1}] \subset \R^n\}$
    where we denote by $[\bbx, \bby]$ the line segment connecting $\bbx$ and $\bby$ (the interval from $\bbx$ to $\bby$), and set $\bbe_0 = \mathrm{0} \in \R^n$. It is easy to see that every vertex of the simplex $\Delta_I$ lies in the zonotope
    \[\Diamond_I := \sum_{\Delta \in \binom{I}{2}^*}\Delta\]
    spanned by the intervals in $\binom{I}{2}^*.$ This implies that $\Delta_I \subseteq \Diamond_I$ for all $I \in K$. In addition, let
    \[B_K := \bigcup_{I \in K} \binom{I}{2}^*\]
    be the set of intervals from all of $\binom{I}{2}^*$ for $I \in K$. Thus, $\hat{\Pi}$ lies in the zonotope
    \[\Diamond_K := \sum_{\Delta \in B_K}\Delta = \sum_{I \in K}\Diamond_I.\]
    Since $P_T$ is $n$-dimensional, by Remark \ref{rem: at-most-n-left-vertices}, $\hat{H}$ is a tree and a bipartite graph with $n+1$ right vertices. In addition, because $|\binom{I}{2}^*| = |I|-1$, it follows that $$\sum_{I \in K}\Big|\binom{I}{2}^*\Big| = n.$$ Hence, the polytope $\Diamond_K$ is an integral translation of a fundamental parallelepiped in $\R^n$. Since $\hat{\Pi} \subseteq \Diamond_K$ and $\Pi$ is an integral translation of $\hat{\Pi}$, it follows that $\Pi$ lies in an integral translation of a fundamental parallelepiped as claimed.
\end{proof}

Let $S \in \ads$ be an admissible subset. We define $\nabla_S$ to be the simplex closely related to $\Delta_S$ by $$\nabla_S:=\conv\left(\sum_{i \in S}\bbe_i,\sum_{i \in J}\bbe_i\ |\ J \subset S \text{ and } |J| = |S|-1\right).$$ 
The next lemma characterizes the fine mixed cells that are zonotopes, i.e., those fine mixed cells whose corresponding bipartite graphs are trees having left vertices of degree at most two. 

\begin{lem}\label{lem: zonotope-equivalence} Let $P = \Delta^0_{S_1} + \cdots + \Delta^0_{S_m}$ where $S_i \in \ads$ for all $i \in [m]$. Suppose that $T \in \ads_n$ is an admissible subset in which $P_T$ is $n$-dimensional. Let $\Pi$ be a fine mixed cell of $P_T$ and $\bbx \in \Z^n$ be the lattice point satisfying $\bbx+ \hat{\Pi} = \Pi$. Then, the following statements are equivalent.
\begin{enumerate}
    \item\label{itm: contains-nabla} The fine mixed cell $\Pi$ satisfies $\nabla_T \subseteq\hat{\Pi}.$
    \item \label{itm: contains-nabla-translation} $\bbx$ is the unique lattice point in $(P_T-\square_T)\cap\Z^n$ such that $\bbx + \nabla_T\subseteq \Pi$.
    \item \label{itm: zonotope} The fine mixed cell $\Pi$ is a zonotope.
    
\end{enumerate}
\end{lem}

To prove this lemma, we need the following result regarding the existence of a perfect matching of a certain bipartite graph.

\begin{lem}\label{lem: perfect-matching}
    Let $n$ be a positive integer. Suppose that $B$ is a bipartite graph on $n$ left vertices and $n+1$ right vertices $\{r_1, \dots, r_{n+1}\}$. Assume further that $B$ is a tree and that every left vertex of it has degree two. Then, there is a perfect matching in the induced bipartite graph obtained by removing the right vertex $r_i$ of $B.$
\end{lem}

\begin{proof} We will proceed by induction on $n \geq 1.$ It is easy to see that the statement is true for $n = 1$.  Now suppose that the statement is true for some $n \geq 1.$ Consider a bipartite graph $B$ on $n+1$ left vertices $\{\ell_1, \dots, \ell_{n+1}\}$ and $n+2$ right vertices $\{r_1, \dots, r_{n+2}\}$ satisfying the assumption of Lemma \ref{lem: perfect-matching}. Since $B$ is a tree and every left vertex of it has degree two, there exist two right vertices of degree one that are not adjacent to the same left vertex. Without loss of generality, we may assume that $r_1$ and $r_{n+2}$ are right vertices of degree one, and that $r_1$ is adjacent to $\ell_1$ and $r_{n+2}$ is adjacent to $\ell_{n+1}$. Let $B'$ be the induced bipartite graph obtained by removing $\ell_{1}$ and $r_{1}$ from $B$, and let $B''$ be the induced bipartite graph obtained by removing $\ell_{n+1}$ and $r_{n+2}$ from $B$. Then, both $B'$ and $B''$ are bipartite graphs on $n$ left vertices and $n+1$ right vertices satisfying the assumption of Lemma \ref{lem: perfect-matching}. Let $B_{(i)}$ be the induced bipartite graph obtained by removing the right vertex $r_i$ from $B$. Similarly, let $B'_{(j)}$ and $B''_{(k)}$ be the induced bipartite graphs obtained by, respectively, removing $r_j$ from $B'$ and removing $r_k$ from $B''$. By the induction hypothesis, both $B'_{(j)}$ and $B''_{(k)}$ contains a perfect matching. Thus, a perfect matching in $B_{(n+2)}$ can be constructed by matching $\ell_1$ with $r_1$ and matching other vertices of $B_{(n+2)}$ using a perfect matching in the $B'_{(n+2)}$. Similarly, for $i \in [n+1]$, one can construct a perfect matching in $B_{(i)}$ by matching $\ell_{n+1}$ with $r_{n+2}$ and matching other vertices of $B_{(i)}$ using a perfect matching in the $B''_{(i)}$. Therefore, the statement holds for all $n \geq 1.$
\end{proof}

\begin{proof}[Proof of Lemma \ref{lem: zonotope-equivalence}] Due to symmetry, we may assume for simplicity of notation without loss of generality that $T = [n]$. Let $\bbx \in \Z^n_{\geq 0}$ be the lattice point satisfying $\bbx+ \hat{\Pi} = \Pi$.

Firstly, we show that \eqref{itm: contains-nabla} implies \eqref{itm: contains-nabla-translation}. Suppose that $\nabla_{[n]} \subseteq\hat{\Pi}.$ Then, $\bbx + \nabla_{[n]} \subset \bbx + \hat{\Pi} = \Pi.$ In particular, we have $\bbx + \bbe_1 + \cdots + \bbe_n \in \Pi \subseteq P_{[n]}$. By Lemma \ref{lem: inP}, we must have $\bbx + \square_{[n]} \subseteq P_{[n]}.$ Thus, $\bbx \in (P_{[n]} - \square_{[n]})\cap\Z^n.$ Hence, $\bbx \in (P_{[n]} - \square_{[n]})\cap\Z^n$ is an integral vector such that $\bbx + \nabla_{[n]} \subset \Pi$.

Suppose that $\bby \in (P_{[n]} - \square_{[n]})\cap\Z^n$ satisfies $\bby + \nabla_{[n]} \subset \Pi.$ Since $\Pi = \bbx + \hat{\Pi}$, we have $\nabla_{[n]} \subseteq \bbx - \bby + \hat{\Pi}.$ Because $\nabla_{[n]} \subseteq\hat{\Pi}$, we deduce that $\hat{\Pi}$ has common interior with $\bbx - \bby + \hat{\Pi}$. By Lemma \ref{lem: no-common-interior}, we must have $\bbx = \bby.$ This shows the uniqueness of the lattice point $\bbx$.

Next, we show that \eqref{itm: contains-nabla-translation} implies \eqref{itm: zonotope}. Suppose that $\bbx \in (P_{[n]}-\square_{[n]})\cap\Z^n$ and $\bbx + \nabla_{[n]}\subseteq \Pi$. We claim that $\Pi$ is a zonotope. Assume for the sake of contradiction that $\Pi$ is not a zonotope. Then, $\hat{\Pi}$ is also not a zonotope. This implies that $\hat{H}$ is a bipartite graph that is also a tree with at least one left vertex having degree at least three. Together with Remark \ref{rem: at-most-n-left-vertices}, we conclude that $\hat{H}$ must have less than $n$ left vertices. Thus, the set $\{r_{1},\dots, r_{n}\}$ of the right vertices cannot be a transversal of the neighbors of the left vertices of $\hat{H}.$ Thus, by Lemma \ref{lem: lattice-points-in-fine-mixed-cell} and Remark \ref{rem: integers-in-fine-mixed-cells}, we see that $\bbe_1 + \cdots + \bbe_n \not\in \hat{\Pi}.$ In particular, we must have $\nabla_{[n]} \not\subseteq \hat{\Pi}.$ Therefore, $\bbx + \nabla_{[n]} \not\subset  \bbx + \hat{\Pi} = \Pi,$ a contradiction. This shows that \eqref{itm: contains-nabla-translation} implies \eqref{itm: zonotope}.

Lastly, we show that \eqref{itm: zonotope} implies \eqref{itm: contains-nabla}. Suppose that $\Pi$ is a zonotope. Then, $\hat{\Pi}$ is also a zonotope. Thus, $\hat{H}$ is a bipartite graph that is also a tree on $n+1$ right vertices in which every left vertex has degree two. By Remark \ref{rem: at-most-n-left-vertices}, $\hat{H}$ must have exactly $n$ left vertices. To see that $\nabla_{[n]} \subseteq \hat{\Pi}$, it suffices to show  $\bbe_1 + \cdots + \bbe_n \in \hat{\Pi}$ and $\bbe_1 + \cdots \bbe_{i-1} + \bbe_{i+1} + \cdots + \bbe_n\in \hat{\Pi}$ for all $i \in [n]$. For $i \in \{0, 1, \dots, ,n\}$, let $\hat{H}_{(i)}$ be the induced bipartite graph obtained by removing the right vertex $r_i$ of $\hat{H}$. Then, by Lemma \ref{lem: perfect-matching}, there is a matching in $\hat{H}_{(i)}$ for all $i \in \{0, 1, \dots, n\}$. By Remark \ref{rem: integers-in-fine-mixed-cells}, a matching of $\hat{H}_{(0)}$ implies that $\bbe_1 + \cdots + \bbe_n \in \hat{\Pi}$ while a matching of $\hat{H}_{(i)}$ implies that $\bbe_1 + \cdots \bbe_{i-1} + \bbe_{i+1} + \cdots + \bbe_n\in \hat{\Pi}$ for all $i \in [n]$. This completes the proof.
\end{proof}

We are now ready to give a proof of Lemma \ref{lem: fine-mixed-subdivision-latticepts}.

\begin{proof}[Proof of Lemma \ref{lem: fine-mixed-subdivision-latticepts}]
     We first show that $|\mathrm{D}(P_T)\cap \square_{[m]}|= |\calC^*|$. Suppose that $\calC^* = \{\Pi_1, \dots, \Pi_q\}$. Since $P_T$ is $n$-dimensional, its corresponding bipartite graph $G(P_T)$ is connected. Moreover, because all fine mixed cells in $\calC^*$ are zonotopes, each cell $\Pi_i \in \calC^*$ corresponds to a spanning tree $H_i$ of $G(P_T)$ whose left vertices have degree at most two. By Lemma \ref{itm: left-deg-g-draconion}/\ref{itm: left-deg-g-draconion}, $LD(H_1), \dots, LD(H_q) $ are distinct elements of $\mathrm{D}(P_T)\cap \square_{[m]}$. Thus, $|\mathrm{D}(P_T)\cap \square_{[m]} \geq |\calC^*|$. Let $\calC$ be the fine mixed subdivision of $P_T$ in which $\calC^* \subseteq \calC$. Also by Lemma \ref{itm: left-deg-g-draconion}/\ref{itm: left-deg-g-draconion}, for every $\bba \in \mathrm{D}(P_T)\cap \square_{[m]}$, there exists a fine mixed cell $\Pi \in \calC$ such that $\bba = LD(H)$ where $H$ is a corresponding spanning tree of $\Pi$. This implies that the corresponding spanning tree of $\Pi$ has left vertices of degree at most two. Thus, the fine mixed cell $\Pi$ is a zonotope, i.e., $\Pi \in \calC^*.$ Therefore, $|\mathrm{D}(P_T)\cap \square_{[m]}| \leq |\calC^*|$, which then implies $|\mathrm{D}(P_T)\cap \square_{[m]}|= |\calC^*|$ as desired.

    Next, we show that $|\calC^*| \leq |(P_T - \square_T)\cap \Z^n|.$ By Lemma \ref{lem: zonotope-equivalence}/\eqref{itm: contains-nabla-translation}, the map $\Pi_i \mapsto \bbx$ where $\bbx \in (P_T - \square_T)\cap \Z^n$ satisfying $\bbx + \nabla_T \subseteq \Pi_i$ is an injection from $\calC^*$ to $(P_T - \square_T)\cap \Z^n$. Hence, $|\calC^*| \leq |(P_T - \square_T)\cap \Z^n|.$

    Lastly, we show that $|\calC^*| \geq |(P_T - \square_T)\cap \Z^n|.$ Let $\bbx \in (P_T - \square_T)\cap \Z^n.$ Due to symmetry, we may assume for simplicity of notation without loss of generality that $T = [n]$. Let $\Pi \in \calC$ be a fine mixed cell that has a common interior with $\bbx + \nabla_{[n]}.$ Note that such a fine mixed cell $\Pi$ exists, because $\bbx + \nabla_{[n]} \subseteq \bbx + \square_{[n]} \subseteq P_T$. We claim that $\Pi$ is a zonotope ($\Pi \in \calC^*$) and that $\bbx$ is the unique lattice point in which $\bbx + \nabla_{[n]}$ is contained in $\Pi$. This claim readily implies that $|\calC^*| \geq |(P_T - \square_T)\cap \Z^n|$, since it establishes the injectivity of the map from $(P_T - \square_T)\cap \Z^n$ to $\calC^*$ defined by $\bbx \mapsto \Pi$ where $\bbx + \nabla_{[n]} \subseteq \Pi$. Thus, it only remains to prove the claim.
    
    Assume for the sake of contradiction that $\Pi$ is not a zonotope. Then, $\hat{\Pi}$ is also not a zonotope. This implies that $\hat{H}$ is a bipartite graph that is also a tree with at least one left vertex having degree at least three. Together with Remark \ref{rem: at-most-n-left-vertices}, we conclude that $\hat{H}$ must have less than $n$ left vertices. Thus, the set $\{r_{i_1},\dots, r_{i_n}\}$ of the right vertices cannot be a transversal of the neighbors of the left vertices of $\hat{H}.$ By Lemma \ref{lem: lattice-points-in-fine-mixed-cell} and Remark \ref{rem: integers-in-fine-mixed-cells}, we deduce that $\hat{\Pi}$ lies in the half-space $x_1 + \cdots + x_n \leq n-1.$ Because $\hat{\Pi} \subseteq \R^n_{\geq 0}$, it follows that $\hat{\Pi}$ lies in the polytope $Q$ given by
\begin{align*}
    Q:= \{\bbx \in \R^n\ | \ x_1 + \cdots + x_n \leq n-1 \ \text{ and } 0 \leq x_i \text{ for all } i \in [n]\} = (n-1)\sym_{[n]}.
\end{align*}
Note that $\nabla_{[n]}:= \{\bbx  \in \R^n \ | \ x_1 + \cdots + x_n \geq n-1 \ \text{ and } 0 \leq x_i \leq 1 \text{ for all } i \in [n]\}.$ Thus, for $\bby = (y_1, \dots, y_n)$ and $\bbz = (z_1, \dots, z_n)$ in $\Z^n,$ the translations $\bby + Q$ and $\bbz+ \nabla_{[n]}$ are given by
\begin{align*}
    \bby + Q&= \{\bbx\ | \ x_1 + \cdots + x_n \leq y_1 + \cdots + y_n + n-1 \ \text{ and } y_i \leq x_i \text{ for all } i \in [n]\}\\
    \bbz+ \nabla_{[n]}&= \{\bbx\ | \ x_1 + \cdots + x_n \geq z_1 + \cdots + z_n + n-1 \ \text{ and } z_i \leq x_i \leq z_i + 1 \text{ for all } i \in [n]\}.
\end{align*}
If $z_1 + \cdots + z_n + n-1 \geq y_1 + \cdots + y_n + n-1$, then $\bby + Q$ and $\bbz + \nabla_{[n]}$ have no common interior. If $z_1 + \cdots + z_n + n-1 < y_1 + \cdots + y_n + n-1$, then there exists $j \in [n]$ such that $z_j + 1 \leq y_j.$ This implies that $\bby + Q$ lies in the half-space $z_j + 1 \leq x_j$ while $\bbz + \nabla_{[n]}$ lies in the half-space $x_j \leq z_j + 1.$ Hence, $\bby+ Q$ and $\bbz + \nabla_{[n]}$ will also have no common interior in this case. That is, one must have that $\bby+ Q$ and $\bbz + \nabla_{[n]}$ have no common interior for all $\bby, \bbz \in \Z^n.$ As a consequence, one has that $\bby + \hat{\Pi}$ and $\bbz + \nabla_{[n]}$ have no common interior for all $\bby, \bbz \in \Z^n.$ Since $\Pi$ is an integral translation of $\hat{\Pi}$, it also follows that $\Pi$ has no common interior with $\bbx + \nabla_{[n]}$, a contradiction to our choice of $\Pi$. Therefore, $\Pi$ is a zonotope as claimed. 

Now suppose that $\bby \in \Z^n$ is the integral vector such that $\bby + \hat{\Pi} = \Pi$. Since $\Pi$ is a zonotope, by Lemma \ref{lem: zonotope-equivalence}/\eqref{itm: contains-nabla-translation}, $\bby$ is the unique lattice point such that $\bby + \nabla_{[n]} \subseteq \Pi.$  Thus, $\bbx + \nabla_{[n]} \subseteq \bbx-\bby+\Pi.$ This means that $\bbx-\bby+\Pi$ and $\Pi$ have a common interior. By Lemma \ref{lem: no-common-interior}, we must have $\bbx = \bby.$ Therefore, $\bbx$ is, as claimed, the unique lattice point such that $\bbx + \nabla_{[n]} \subseteq \Pi$.
\end{proof}

The next theorem is a consequence of Lemma \ref{lem: fine-mixed-subdivision-latticepts}.

\begin{thm}\label{thm: lattice-one-to-one} Suppose that $P = \Delta^0_{S_1} + \cdots + \Delta^0_{S_m}$ where $S_i \in \ads$ for all $i \in [m]$. Then, 
\begin{align}\label{eq: lattice-pts-typeB}
    |(P - \square_{[n]})\cap \Z^n| = \sum_{T \in \ads_n}|\mathrm{D}(P_T)\cap \square_{[m]}|.
\end{align}
\end{thm}

\begin{proof}
We first observe that 
\[(P - \square_{[n]})\cap \Z^n = \bigsqcup_{T\in \ads_n}(P_T-\square_{[n]})\cap\Z^n.\]
Thus,
\[|(P - \square_{[n]})\cap \Z^n| = \sum_{T \in \ads_n}|(P_T - \square_{[n]})\cap \Z^n|.\]
Since $\square_{T}$ is a translation of $\square_{[n]}$ by an integral vector, we have $|(P_T - \square_{[n]})\cap \Z^n| = |(P_T - \square_{T})\cap \Z^n|$ for all $T \in \ads_n$. Consequently, we have 
\begin{align}\label{eq: lattice-pts-typeB2}
    |(P - \square_{[n]})\cap \Z^n| = \sum_{T \in \ads_n}|(P_T - \square_{T})\cap \Z^n|.
\end{align}

For $T \in \ads_n$ such that $P_T$ is not $n$-dimensional, we have that $|(P_T - \square_{T})\cap \Z^n| = 0 = |\mathrm{D}(P_T)\cap \square_{[m]}|$. Also, for $T \in \ads_n$ such that $P_T$ is $n$-dimensional, we have by Lemma \ref{lem: fine-mixed-subdivision-latticepts} that $|(P_T - \square_{T})\cap \Z^n| = |\mathrm{D}(P_T)\cap \square_{[m]}|$. Thus, formula \eqref{eq: lattice-pts-typeB} follows from \eqref{eq: lattice-pts-typeB2}.
\end{proof}

Theorem \ref{thm: lattice-one-to-one} implies that
\begin{equation}\label{eq: lattice-point-typeb1}
    |(P - \square_{[n]})\cap \Z^n| = \sum_{T \in \ads_n}\left(\sum_{\bba \in \mathrm{D}(P_T)}\binom{1}{a_1}\cdots \binom{1}{a_m}\right),
\end{equation}
since only those $G$-draconian sequences $\bba = (a_1, \dots, a_m)$ with $a_i \leq 1$ make the summands in equation (\ref{eq: lattice-point-typeb1}) nonzero, and equal to one. Together with a simple binomial identity, we derive the following key result as a consequence.
\begin{cor}\label{cor: ehrhart-typeb}
    Suppose that $P = y_1\Delta^0_{S_1} + \cdots + y_m\Delta^0_{S_m}$ where $y_i$ are integers and $S_i \in \ads$ for all $i \in [m]$. Then, the number of lattice points in $P - \square_{[n]}$ is given by
    \begin{equation}\label{eq: lattice-points-typeb2}
    |(P - \square_{[n]})\cap \Z^n| = \sum_{T \in \ads_n}\left(\sum_{\bba \in \mathrm{D}(P_T)}\binom{y_1}{a_1}\cdots \binom{y_m}{a_m}\right).
\end{equation}
\end{cor}

\begin{proof} As noted at the beginning of the section, it suffices to show that the formula holds for positive integers $y_1, \dots, y_m$. Clearly, we can write any $y\Delta^0_{S}$ where $S \in \ads$ and $y$ is a positive integer as the Minkowski sum of $y$ copies of $\sym_S.$ By writing 
\[P = \underbrace{\sym_{S_1} + \cdots +\sym_{S_1}}_{y_1 \text{ terms }} + \cdots + \underbrace{\sym_{S_m} + \cdots + \sym_{S_m}}_{y_m \text{ terms }}\]
and applying formula \eqref{eq: lattice-point-typeb1} to $P$, one can express the right-hand side of \eqref{eq: lattice-point-typeb1} as
\[\sum_{T \in \ads_n}\left(\sum_{\bba \in \mathrm{D}(P_T)}\binom{y_1}{a_1}\cdots \binom{y_m}{a_m}\right)\]
using the binomial identity
\[\binom{y}{a} = \sum_{\underset{b_1, \dots,\ b_y \,\in \Z_{\geq 0}}{b_1 + \cdots + b_y = a}} \binom{1}{b_1}\cdots \binom{1}{b_y} \ \ \text{ for all } a,y \in \Z_{\geq 0}.\qedhere\]
\end{proof}

\begin{rem}\label{rem: formula-compared-to-Eurs}
     Let $$D(P) := \bigcup_{T \in \ads_n}D(P_T).$$ Then, formula \eqref{eq: lattice-points-typeb2} can also be expressed as
    \begin{align}\label{eq: ehrhart-type-b-2}
        |(P - \square_{[n]})\cap \Z^n| 
        &= \sum_{\bba \in D(P)}|\{T \in \ads_n \ | \ \bba \in D(P_T)\}|\binom{y_1}{a_1}\cdots \binom{y_m}{a_m}.
    \end{align}
    One observes that formula \eqref{eq: ehrhart-type-b-2} bears a resemblance to formula \eqref{eq: typeB-Ehrhart-Eur} from Lemma \ref{lem: typeB-Vol-Ehrhart} by Eur, Fink, Larson, and Spink. However, \eqref{eq: ehrhart-type-b-2} is written in a more compact form: each summand in \eqref{eq: ehrhart-type-b-2} consolidates $\frac{(a_1 + \cdots + a_m)!}{a_1! \cdots a_m!}$ individual terms that appear in   \eqref{eq: typeB-Ehrhart-Eur}.
\end{rem}

The Ehrhart polynomial of $P = y_1\Delta^0_{S_1} + \cdots + y_m\Delta^0_{S_m}$ can be computed simply by replacing $P$ in \eqref{eq: lattice-points-typeb2} with $tP + \square_{[n]} = ty_1\Delta^0_{S_1} + \cdots + ty_m\Delta^0_{S_m} + \sym_{\{1\}} + \cdots + \sym_{\{n\}}$. This allows us to obtain a formula for the volume of $P$ by computing the leading coefficient of its Ehrhart polynomial as stated in the following corollary.
\begin{cor}\label{cor: volume-typeb}
    Suppose that $P = y_1\Delta^0_{S_1} + \cdots + y_m\Delta^0_{S_m}$ where $y_i$ are real numbers and $S_i \in \ads$ for all $i \in [m]$. Then, the volume of $P$ is given by
    \vspace{-0.5cm}\begin{equation}\label{eq: volume-typeb}
    \vol(P) = \sum_{T \in \ads_n}\left(\sum_{\bba \in \mathrm{D}(P_T)}\frac{y_1^{a_1}}{a_1!}\cdots \frac{y_m^{a_m}}{a_m!}\right).
\end{equation}
\end{cor}

\begin{proof}
    To see that formula \eqref{eq: volume-typeb} gives the volume of $P$, it suffices to show that the formula holds for nonnegative integers $y_1, \dots, y_m$. Let $$Q^{(t)} := tP + \square_{[n]} = ty_1\Delta^0_{S_1} + \cdots + ty_m\Delta^0_{S_m} + \sym_{\{1\}} + \cdots + \sym_{\{n\}}.$$ We note that $\mathrm{D}(Q^{(t)}_T) \subseteq \Z^{m+n}_{\geq 0}$. Then, the Ehrhart polynomial of $P$ is given by
    \[i(P,t) = |(Q^{(t)} - \square_{[n]})\cap \Z^n| = \sum_{T \in \ads_n}\left(\sum_{\bba \in \mathrm{D}(Q^{(t)}_T)}\binom{y_1t}{a_1}\cdots \binom{y_mt}{a_m}\binom{1}{a_{m+1}}\cdots \binom{1}{a_{n+m}}\right).\]
    Since the volume of $P$ equals the leading coefficient of $i(P,t)$, it follows that
    \begin{align}\label{eq: volume-typeb2}
        \vol(P) = \sum_{T \in \ads_n}\left(\sum_{\underset{a_{i} = 0, \ \forall\, i\, >\, m}{\bba \in \mathrm{D}(Q^{(t)}_T)}}\frac{y_1^{a_1}}{a_1!}\cdots \frac{y_m^{a_m}}{a_m!}\right).
    \end{align}
    One observes that $\bba \in \mathrm{D}(Q^{(t)}_T)$ satisfies $a_{i} = 0$ for all $i\, >\, m$ if and only if $(a_1, \dots, a_m) \in \mathrm{D}(P_T)$. Thus, the volume formula \eqref{eq: volume-typeb} is equivalent to \eqref{eq: volume-typeb2}.
\end{proof}

\section{More Problems from A to B}

Our approach for computing the number of lattice points suggests that there seem to be many aspects of type B generalized permutohedra that can be explored using existing techniques and tools from the study of their type A counterparts. The following problems highlight some potential research directions.
\begin{op}
    In \cite[Section 7]{Postnikov2005}, Postnikov introduces building sets and nested complexes to describe the face posets of some type A generalized permutohedra. Can we give a combinatorial description of the faces of type B generalized permutohedra using similar combinatorial models as building sets and nested complexes?
\end{op}

\begin{op}
    In \cite{PostnikovEtal2008}, Postnikov, Reiner, and Williams compute the $f$ and $h$-polynomials of a family of simple type A generalized permutohedra using building sets and the corresponding preposets.  Can we employ a similar approach to compute the $f$ and $h$-polynomials of type B generalized permutohedra?
\end{op}

\begin{op}
    Bastidas shows in \cite{Bastidas2021} that every type B generalized permutohedron can also be written as the Minkowski sum of the simplices $\Delta_S$ and $\sym_S$ where $S \in \ads$ are admissible subsets such that $\min(|i|\ | \ i \in S) \in S.$ That is, the family of admissible subsets $\Delta_S$ and $\sym_S$ where $\min(|i|\ | \ i \in S) \in S$ is a ``basis'' for the type B generalized permutohedra. Find a formula for the Ehrhart polynomial of type B generalized permutohedra with respect to this basis (directly without transforming this basis to the basis used in this paper).
\end{op}

\begin{op}
    An integral polytope is said to be \emph{Ehrhart positive} if every coefficient of its Ehrhart polynomial is positive. Postnikov's formula \eqref{eq: lattice-type-A} implies that every type $A$ generalized permutohedron of the form $y_1\Delta_{I_1} + \cdots + y_m\Delta_{I_m} \subset \R^n$ is Ehrhart positive provided $y_i$ are positive integers for all $i\in [m]$. Our formula \eqref{eq: lattice-points-typeb2} in Corollary \ref{cor: ehrhart-typeb} does not make it immediately clear whether Ehrhart positivity holds for type $B$ generalized permutohedra in a similar situation. This leads to the question: Is every type $B$ generalized permutohedron of the form $y_1\sym_{S_1} + \cdots + y_m\sym_{S_m} \subset \mathbb{R}^n$ Ehrhart positive provided $y_i$ are positive integers for all $i \in [m]$?
\end{op}




\bibliographystyle{abbrv}
\bibliography{BibContainer}

\end{document}